\documentclass{amsproc}

\newtheorem{theorem}{Theorem}[section]
\newtheorem{lemma}[theorem]{Lemma}

\theoremstyle{definition}

\newtheorem{example}[theorem]{Example}

\theoremstyle{remark}
\newtheorem{remark}[theorem]{Remark}

\numberwithin{equation}{section}

\usepackage{tcolorbox}

\usepackage{comment}

\usepackage[abs]{overpic}		  

\usepackage{xcolor}
\definecolor{indigo}{rgb}{0.29, 0.0, 0.51}  
\definecolor{lred}{RGB}{226, 106, 106}
\definecolor{dgreen}{RGB}{0, 161, 75}
\definecolor{dblue}{RGB}{33, 64, 154}
\definecolor{lblue}{RGB}{0, 174, 239}
\usepackage[colorlinks, urlcolor=indigo, linkcolor=indigo, citecolor=indigo]{hyperref}

\newtheorem{corollary}[theorem]{Corollary}
\newtheorem{exercise}[theorem]{Exercise}
\theoremstyle{remark}

\newcommand{\bhw}{\begin{exercise}} 
\newcommand{\ehw}{\end{exercise}}
\newcommand{\bex}{\begin{example}} 
\newcommand{\eex}{\end{example}}

\def\co{\colon\thinspace}

\numberwithin{theorem}{section}

\newcommand{\dfn}[1]{{\em #1}}        
\newcommand{\R}{\mathbb{R}}           
\newcommand{\Z}{\mathbb{Z}}           
\newcommand{\C}{\mathbb{C}}           

\makeatletter
\newcommand*\bigcdot{\mathpalette\bigcdot@{0.6}}
\newcommand*\bigcdot@[2]{\mathbin{\vcenter{\hbox{\scalebox{#2}{$\m@th#1\bullet$}}}}}
\makeatother

\DeclareMathOperator\idx{index}  
\DeclareMathOperator\img{image}  

\DeclareMathOperator{\sym}{{Sym}} 

\begin{document}

\title{Building homology theories (ala Floer)} 

\author{John B. Etnyre}
\address{School of Mathematics \\ Georgia Institute
of Technology \\  Atlanta, Georgia}
\email{etnyre@math.gatech.edu}
\thanks{The first author was supported in part by NSF Grant DMS-2203312 and the Georgia Institute of Technology Elaine M. Hubbard Distinguished Faculty Award.}


%
%

\begin{abstract}
These notes are an expanded version of evening talks at the 2025 Georgia International Topology Conference, and an abbreviated version of talks at Georgia Tech, which were aimed at graduate students. The hope was to indicate a common framework that has been used since the late 1980s to construct homology theories in low-dimensional topology and symplectic and contact geometry. In addition to this, we also try to indicate how the specific nature of the situation being studied dictates the algebraic nature of the chain groups used to define the homology. 
\end{abstract}

\maketitle
\setcounter{tocdepth}{2}
\tableofcontents

\section{Introduction}

In the last 40 or so years, there have been lots of ``homology theories" in topology and geometry that are built through a similar process. This goes back to work of Floer in \cite{Floer1987, Floer1988} that built Symplectic Floer Homology and Lagrangian Floer Homology by looking at the ``moduli space'' of solutions to a partial differential equation. Since then, the classical Morse Homology has been reinterpreted in this manner, see \cite{Schwarz1993}, and countless other homology theories have been defined. For example, various flavors of Instanton Floer Homology \cite{Donaldson2002, Floer1988a, KronheimerMrowka2010} and Monopole Floer Homology \cite{KronheimerMrowka2007} are defined in a similar way, and Contact Homology \cite{Eliashberg98}, Embedded Contact Homology \cite{Hutchings2002, Hutchings2014}, Legendrian Contact Homology \cite{Chekanov02, EkholmEtnyreSullivan07, Eliashberg98}, Heegaard Floer Homology \cite{OzsvathSzab02004}, and Knot Contact Homology \cite{EkholmEtnyreNgSullivan2013} are defined using variants of the Floer approach. (To be clear, the work to define each of these homology theories is considerable. While the approach was laid out by Floer, there are considerable details to be understood and dealt with.)

In these notes, we will discuss a general procedure for constructing the homology theories above. We will illustrate the process in three examples. The simplest case, with which we start, is Morse Homology. We then discuss Lagrangian Floer Homology and Legendrian Contact Homology. The main purpose of this paper is to (1) demonstrate the general process for constructing such homology theories and (2) give a good indication of how one must choose the correct algebraic structure for the chain groups of the homology theory based on the geometry of the problem at hand. There are several great references for the details of Morse Homology \cite{HutchingsNotes, Schwarz1993}, and some for the other theories \cite{AbbondandoloSchlenk2019, AudinDamian2014, BanyagaHurtubise2004}, but we hope by not going into a lot of detail on the proofs (though Section~\ref{details} gives a flavor for these), we can help people see the general approach to building homology theories and how ``compactness issues'' force the necessary algebraic structures for the homology theory. 

Before diving into the examples, we outline a general process for constructing invariants. 
Any unfamiliar terms in the examples below will be explained later in the text. 

\begin{tcolorbox}[title={\bf Step 0:} Start with something you are interested in studying.]
Examples:
\begin{itemize}
\item For Morse homology, we want to study manifolds.
\item For Lagrangian Floer homology, we want to study pairs of Lagrangian submanifolds of symplectic manifolds.
\item For Legendrian contact homology, we want to study Legendrian submanifolds of contact manifolds. 
\end{itemize}
\end{tcolorbox}

\begin{tcolorbox}[title={\bf Step 1:} Choose extra data.]
Examples:
\begin{itemize}
\item For Morse homology, we choose a pair $(f,g)$ where $f\co M\to \R$ is a Morse function and $g$ is a Riemannian metric on $M$. 
\item For Lagrangian Floer homology, we choose an almost-complex structure compatible with the symplectic structure. 
\item For Legendrian contact homology, we choose a contact form $\alpha$ for the contact structure and a complex structure compatible with $d\alpha$ on the contact hyperplanes. 
\end{itemize}
\end{tcolorbox}

\begin{tcolorbox}[title={{\bf Step 2:} Find an algebraic object, usually a chain complex, that counts something of geometric interest.}]
Examples:
\begin{itemize}
\item For Morse homology, choose a chain complex that is a vector space generated by critical points of $f$ and the differential counts gradient flow lines of $f$ with respect to $g$. 
\item For Lagrangian Floer homology, choose a chain complex that is a vector space generated by intersection points between the two Lagrangian submanifolds and the differential counts holomorphic disks ``between intersection points''. 
\item For Legendrian contact homology, choose a chain complex that is an algebra generated by ``Reeb cords" and the differential counts holomorphic disks with certain boundary conditions. 
\end{itemize}
\end{tcolorbox}

\begin{tcolorbox}[colupper=white, colback=black!75!white,colframe=black!75!white]
{{\bf Step 3:} Show something about the computation in Step~2, usually the homology of the complex, does not depend on the extra data chosen in Step~1.}
\end{tcolorbox}

Before getting into the more complicated examples, we give a couple of very simple examples of the above process. 
\bex
We will be interested in the topology of a smooth manifold $M$. We carry out the above procedure as follows. 
\begin{enumerate}
\item[Step 1.] Choose a ``generic" section of the cotangent bundle $T^*M$.
\item[Step 2.] Define a number by the signed count of the zeros of the generic section. 
\item[Step 3.] Show that this integer is independent of the ``generic" section. 
\end{enumerate}
We call the number defined above the \dfn{Euler characteristic} of $M$. 
\eex
\bhw
Rigorously go through Steps~1.\ to~3.\ to show that you have a well-defined integer associated with $M$. If you know another definition of the Euler characteristic of a smooth manifold, show that it agrees with the definition above. 
\ehw
\bex
Suppose we are interested in the topology of knots in $\R^3$. We carry out the above procedure as follows. 
\begin{enumerate}
\item[Step 1.] Choose a ``generic" projection of the knot $K$ to $\R^2$ and remember over and undercrossings in the projection. (That is, choose a knot diagram for $K$.)
\item[Step 2.] Say the diagram $D$ of $K$ is \dfn{$3$-colorable} if the strands of the diagram (that is the arcs the projection is broken into, when small neighborhoods of the crossings are removed from the undercrossing stand) can be colored by three colors so that at each crossing either one or all three colors are used and at least two colors are used. 
\item[Step 3.] Show that being $3$-colorable is independent of the diagram chosen for $K$.
\end{enumerate}
\eex
\bhw
Rigorously go through Steps~1.\ to~3.\ to show that being $3$-colorable is an invariant of a knot $K$ in $\R^3$. \\ Hint: Use Reidemeister moves. We note that the Reidemeister moves simply described a ``generic" path of projection of $K$. 
\ehw
\begin{remark}
We note that there are lots of invariants of knots that can be defined diagrammatically, and so we see the process above to construct invariants is quite robust for knots. (We also note that many of these diagrammatic invariants --- including $3$-colorability --- have other, more intrinsic, definitions that do not come from the above process.)
\end{remark}

While this is a ``high-level" look at defining invariants, we indicate how Steps~1, 2, and~3 go for ``Floer-type homologies". Above, we indicated what the chain groups would be, so we now need to consider what the boundary maps in the chain complex will be. 
This will necessarily be fairly abstract, but we will get into the details of the above homology theories in the following sections. 

To carry out Step~1 and~2 we let $\mathcal{M}^a_b$ be the space of ``objects" counted in the differential (for example, for Morse homology they will be gradient flow lines from $a$ to $b$, for Lagrangian Floer homology they will be holomorphic disk ``from" $a$ ``to'' $b$, and for Legendrian contact homology they are a bit more complicated). The first theorem we need is a transversality theorem that takes the following form. 
\begin{tcolorbox}[title={Theorem 1 (Transversality)}]
For generically chosen data $\mathcal{M}_b^a$ is a manifold of an easily computed dimension. 
\end{tcolorbox}
To see why this is called a transversality theorem, we note that $\mathcal{M}_b^a$ can be defined as the preimage under some function $F$ defined on some function spaces of a submanifold of the target space. So if $F$ is transverse to this submanifold (in some appropriate sense), the preimage will be a manifold. Next, we would like to know that $\mathcal{M}_b^a$ can be oriented.
\begin{tcolorbox}[title={Theorem 2 (Orientability)}]
One can make a priori choices so that $\mathcal{M}^a_b$ can be given an orientation for all choices of $a$ and $b$. 
\end{tcolorbox}
We would also like to know if $\mathcal{M}_b^a$ is compact and if not, how it fails to be compact.
\begin{tcolorbox}[title={Theorem 3 (Compactness)}]
If $\{u_i\}$ is a sequence in $\mathcal{M}_b^a$ that has no subsequence that converges to a point in $\mathcal{M}_b^a$, then specify what some subsequence does ``converge to".
\end{tcolorbox}
We first note that this theorem should really be considered a ``lack of compactness" theorem since frequently the spaces are not actually compact, but it is still common to call this a ``compactness" theorem. 
In all relevant situations, we will see that when $\mathcal{M}^a_b$ has dimension $0$, then $\mathcal{M}_b^a$ is compact. Thus $\mathcal{M}_b^a$ will be a finite number of points with signs, and we can define the boundary map for the chain complex to be 
\[
\partial a= \sum |\mathcal{M}^a_b| b,
\]
where the sum is over all $b$ such that $\mathcal{M}^a_b$ has dimension $0$ and $|\mathcal{M}^a_b|$ is the signed count of points in $\mathcal{M}^a_b$ (we note if you only use the modulo two count of points in $\mathcal{M}^a_b$ then one can skip Theorem~2 above). We will see below that the precise form of the compactness theorem will dictate the exact algebraic structure of the chain complex. The above theorems and comments indicate that $\partial a$ is well-defined. 

To show that $\partial$ is really a boundary map, that is, that $\partial \circ \partial=0$, we need one last type of theorem. 
\begin{tcolorbox}[title={Theorem 4 (Gluing)}]
If we have some ``element" that might be the limit of a sequence in Theorem~3 above, then it is actually such a limit, and in a nice way.
\end{tcolorbox}
In this abstract setting, it is not so clear how this will imply that $\partial \circ \partial=0$, but below, in specific examples, we will see that this is indeed the case. 

So the above four theorems are exactly what is needed to prove that we have a chain complex! To carry out Step~3, that is to see that the homology of this complex does not depend on the extra data chosen in Step~1, we need to prove analogs of Theorems~1 through~4 when one considers ``generic" $1$-parameter families of extra data from Step~1. 

While the above discussion is a bit abstract, in the next section, we will see precise versions of the above theorems and prove that they allow us to define chain complexes whose homology is an invariant of manifolds. 

We have tried to make these notes largely self-contained and accessible to anyone with a background in differential topology at the level of the standard textbooks \cite{GuilleminPollack10, Lee13}. 

\subsection*{Acknowledgements} We thank B\"ulent Tosun and an anonymous referee for valuable feedback on early drafts of the paper. The author was partially supported by the National Science Foundation grant DMS-2203312 and the Georgia Institute of Technology Elaine M. Hubbard Distinguished Faculty Award

\section{Morse Homology}
In this section, we will study a compact manifold without boundary. We do not really need these assumptions on the manifold, but it is easier on a first pass. Once the reader has understood this case, they are encouraged to consider the necessary modifications for manifolds with boundary and non-compact manifolds. But for now
\[
\text{\bf all manifolds in this section are compact and without boundary.}
\]

In the first section, we will recall facts about Morse functions and gradient flows, while in the second section, we will discuss the space of gradient flow lines. In the next section, we will define Morse homology, and in the final section, we will briefly indicate what goes into the proofs of the main theorems. 

We would also like to point out that the approach to Morse homology given below fits nicely with our discussion of other homology theories, but is not the original way in which Morse homology was defined. There were several previous approaches that were much more in the spirit of topology and dynamics. We refer the reader to the beautiful exposition of Bott \cite{Bott1988} on the history of Morse homology. 

\subsection{Morse functions and gradient flows}
We recall that a function
\[
f\co M\to \R
\]
 is called a \dfn{Morse function} if any critical point $p\in M$ of $f$ --- that is the points where $df_p=0$ --- is non-degenerate. Here by \dfn{non-degenerate} we mean that the \dfn{Hessian of $f$ at $p$}, 
 \[
 Hf_p\co T_pM\times T_pM\to \R\co (v,w)\mapsto w\cdot(v\cdot f),
 \]
 is a non-degenerate symmetric $2$-form. Here we think of a vector field as a derivation, so $v\cdot f$ is another function that can be thought of as the directional derivative of $f$ at each point.
 \begin{remark}
 We note that, as defined above, the Hessian is {\em not} well-defined. This is because for a derivation $v\in T_pM$, we have that $v\cdot f$ is a real number, not a function on $M$, so we cannot apply $w$ to this. What we mean by the notation above is given $v\in T_pM$, we extend $v$ to a vector field $\widetilde{v}$ on $M$ so that $\widetilde{v}\cdot f$ is another function on $M$ and then $w\cdot (\widetilde{v}\cdot f)$ makes sense. But the value, in general, depends on the extension $\widetilde{v}$ of $v$. 
 
 Now at a critical point $p$ of $f$ we will see that $Hf_p$ is well-defined. To this end, let $\widetilde{w}$ be an extension of $w$ to a vector field on $M$ and recall that 
 \[
 v\cdot f = df_p(v).
 \]
 So we see that
 \[
(\widetilde{v}\cdot (\widetilde{w}\cdot f))(p)-(\widetilde{w}\cdot(\widetilde{v}\cdot f))(p)= ([\widetilde{v}, \widetilde{w}])_p\cdot f = df_p([\widetilde{v}, \widetilde{w}])=0
 \]
 since $df_p=0$ at a critical point. Here $[\widetilde{v}, \widetilde{w}]_p$ is the Lie bracket of the vector fields evaluated at $p$. Thus, we see that
 \[
 (\widetilde{v}\cdot(\widetilde{w}\cdot f))(p)=(\widetilde{w}\cdot(\widetilde{v}\cdot f))(p)
 \]
 and so
 \[
 (\widetilde{v}\cdot(\widetilde{w}\cdot f))(p)= d(\widetilde{w}\cdot f)_p(\widetilde{v})=d(\widetilde{w}\cdot f)_p({v})
 \]
 since $d(\widetilde{w}\cdot f)_p(\widetilde{v})$ only depends on the value of $\widetilde{v}$ at $p$. We can similarly see $ (\widetilde{v}\cdot(\widetilde{w}\cdot f))(p)$ only depends on the value of $\widetilde{w}$ at $p$ and thus $Hf_p(v,w)$ is well-defined when $p$ is a critical point. We also notice that $Hf_p$ is symmetric at a critical point. 
 \end{remark}
\bhw
In local coordinates $(x_1, \ldots, x_n)$ near $p$, show that  the symmetric, bilinear form $Hf_p$ can be given by the matrix
\[
\begin{pmatrix}
\frac{\partial^2 f}{\partial x_i\partial x_j}
\end{pmatrix}_{i,j=1}^n
\]
in the basis $\frac{\partial}{\partial x_1}, \ldots, \frac{\partial}{\partial x_n}$ for the tangent space given by the coordinate system.
\ehw
\bhw
Show that the matrix expression for $Hf_p$ is independent of local coordinates at a critical point. (That is, show that in another local coordinate system $(y_1, \ldots, y_n)$ you will get a different matrix, but it will be related by a change of basis for the tangent space given by the change of coordinates.)
\ehw
Here is another way to define the Hessian.
\bhw
Suppose that $g$ is a Riemannian metric on $M$ and let $\nabla$ be the Levi-Civita connection on $T^*M$ associated to $g$ (we note that the Levi-Civita connection is usually defined on $TM$, but as $TM$ and $T^*M$ can be identified using $g$ we will consider the connection on $T^*M$). We note that one may actually use any connection on $T^*M$, but this is a natural one. We recall that $\nabla$ allows one to differentiate sections of $T^*M$. So, given a vector field $w$, we obtain the $1$-form $\nabla_w df$. Show that at a critical point $p$ of $f$, $\nabla_wdf$ is independent of the connection and $Hf_p(v,w)=(\nabla_w df)(v)$. 
\ehw
\bhw
Note that $df\co M\to T^*M$ is a section of the cotangent bundle. Show that $f$ is a Morse function if and only if $df$ is transverse to the zero section $Z$ of $T^*M$. 
\ehw
The first thing one would like to know is that any manifold has a Morse function. 
\begin{theorem}
A generic smooth function $f\co M\to \R$ is a Morse function. 
\end{theorem}
In this theorem, by generic, we mean that the space of Morse functions $\mathcal{M}(M)$ on $M$ as a subset of the set $C^\infty(M)$ of all $C^\infty$ functions on $M$ is dense. (Or more precisely, $\mathcal{M}(M)$ is residual, meaning it is the intersection of countably many open dense sets in $C^\infty(M)$.) 

There are many proofs of the above theorem. Most involve embedding $M$ into some Euclidean space, but see \cite[Proposition~6.13]{GolubitskyGuillemin1973} for an alternate and beautiful proof using jet transversality.

Given a Morse function $f\co M\to \R$ on an $n$-manifold $M$ and a critical point $p\in M$, the Hessian $Hf_p$ can be represented by a symmetric $n\times n$ matrix. Thus, the matrix can be diagonalized. This implies that 
\[
T_pM = V_0\oplus V_+\oplus V_-
\]
where $V_+,$ respectively $V_-$, is the subspace of $T_pM$ on which $Hf_p$ is positive, respectively negative, definite and $V_0$ is the subspace where $Hf_p$ is degenerate. Since $f$ is Morse and $p$ is a non-degenerate critical point, we see that $V_0=\{0\}$, and so 
\[
n= \dim T_pM = \dim V_+ + \dim V_-
\]
We call the dimension of $V_-$ the \dfn{index} of $p$. 
\begin{theorem}[Morse Lemma]\label{ML}
Let $p$ be a non-degenerate critical point of 
\[
f\co M\to \R.
\]
If $k$ is the index of $p$ then there is a coordinate chart $\phi\co U\to V$ around $p$ such that 
\[
f\circ \phi^{-1}(x_1,\ldots, x_n)= f(p) -x_1^2-\cdots - x^2_k+x^2_{k+1}+\cdots + x^2_n.
\]
\end{theorem}
\begin{remark}
Recall from differential topology that if the first derivative of a function $f\co M\to \R$ is non-degenerate at a point $p$, that is $df_p\not=0$, then there are local coordinates where $f$ takes the form
\[
\R^n\to \R\co (x_1, \ldots, x_n)\mapsto x_1.
\]
Said another way, if the first derivative of $f$ at $p$ is non-degenerate, then $f$ is locally equivalent to the ``simplest" function with non-degenerate first derivative. 

One may think of the Morse lemma as a second derivative version of this. So if $df_p=0$ but $Hf_p$ is non-degenerate (that is the total second derivative is non-degenerate), then $f$ is locally equivalent to the ``simplest" function with vanishing first derivative and non-degenerate second derivative. 
\end{remark}
There are several proofs of the Morse Lemma; see, for example, \cite{AudinDamian2014, GolubitskyGuillemin1973, Milnor1963}. 

Given a Riemannian metric $g$ on a manifold $M$ and a function $f\co M\to \R$, the \dfn{gradient} of $f$ is the unique vector field $\nabla f$ such that 
\[
g(\nabla f, v)= df(v)
\]
for all vector fields $v$. 
\bhw
Prove that $\nabla f$ exists and is unique. \\Hint: Recall that $g$ is non-degenerate and so the map $T_pM\to T^*_pM\co v\mapsto \iota_v g$ is an isomorphism for all $p$, where $\iota_vg$ is the contraction of $v$ into $g$. That is, $\iota_v g$ is the $1$-form that takes a vector $w$ and gives the real number $g(v,w)$. 
\ehw
The \dfn{gradient flow} of $f$ is the flow of $-\nabla f$. Specifically, given a point $x\in M$ the \dfn{gradient flow line of $f$ through $x$} is 
\[
\gamma: \R\to M
\]
such that 
\begin{align*}
\gamma'(t)&=-\nabla f_{\gamma(t)}\\
\gamma(0)&=x.
\end{align*}
Since we are considering compact manifolds without boundary, the flow line is defined for all time. 

Recall that the gradient of a function points in the direction in which $f$ is increasing the fastest. Thus, the gradient flow is going in the direction of decreasing values of $f$. So if our manifold is the graph of some function $h$ and we think of $f$ as a height function on a manifold, then one can think of the gradient flow as the flow of water poured on the manifold. 

\bex
Consider the manifold $\R^n$ and the function 
\[f(x_1,\ldots, x_n)=-x_1^2-\cdots -x_k^2+x_{k+1}^2+\cdots x_n^2.
\]
We also choose the Euclidean metric 
\[
g=\sum_{i=1}^n dx_i\otimes dx_i.
\]
One easily computes that 
\[
df = -\sum_{i=1}^k 2x_i\, dx_i + \sum_{i=k+1}^n 2x_i\, dx_i
\]
and thus
\[
\nabla f = 2\left(- \sum_{i=1}^k x_i\frac \partial{\partial x_i} + \sum_{i=k+1}^n x_i \frac \partial {\partial x_i}\right).
\]
The gradient flow lines are indicated in Figure~\ref{fig:gf}. 
\begin{figure}[htb]{
\begin{overpic}
{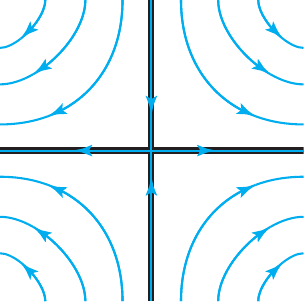}
\put(150, 69){$\R^k$}
\put(68, -12){$\R^{n-k}$}
\end{overpic}}
\caption{The gradient flow of $f(x_1,\ldots, x_n)=-x_1^2-\cdots -x_k^2+x_{k+1}^2+\cdots x_n^2$.}
\label{fig:gf}
\end{figure}  
\eex
\bex\label{ex:sphere}
Consider the unit sphere $S^2$ in $\R^3$ with the metric induced from the Euclidean metric on $\R^3$. If the function $f\co S^2\to \R$ is given by projection to the last coordinate, then we see that there are exactly two critical points $N=(0,0,1)$ and $S=(0,0,-1)$. It is clear that $N$ has index $2$ and $S$ has index $0$. Moreover, the gradient flow is shown in Figure~\ref{fig:sphere}. 
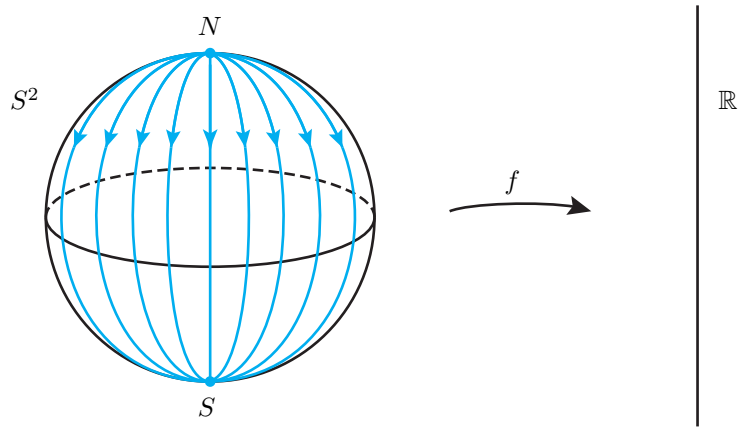
\begin{figure}[htb]{
\begin{overpic}
{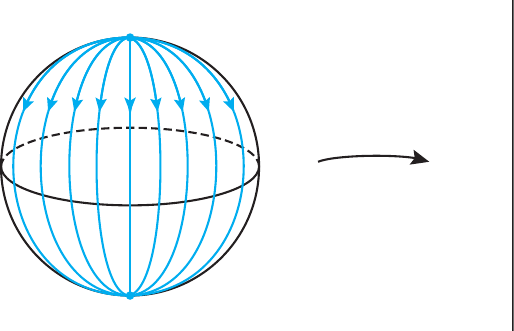}
\put(58, 148){$N$}
\put(58, 4){$S$}
\put(174, 90){$f$}
\put(255, 120){$\R$}
\put(-13, 120){$S^2$}
\end{overpic}}
\caption{The gradient flow of the height function on the sphere.}
\label{fig:sphere}
\end{figure}  
\eex
\bhw
Show that a flow line of the gradient in the previous example through the point $x=(\cos\theta,\sin\theta,0)$ is given by
\[
\gamma(t)=\left( \frac{\cos\theta}{\cosh t}, \frac{\sin\theta}{\cosh t}, \tanh t\right).
\]
Hint: Show that in stereographic coordinates given by 
\[
\phi(u,v)=\frac{1}{1+u^2+v^2}(2u, 2v, u^2+v^2-1)
\]
the metric is given by 
\[
\frac{4}{(1+u^2+v^2)^2}(du\otimes du+ dv\otimes dv)
\]
or in polar coordinates 
\[
\frac{4}{(1+r^2)^2}(dr\otimes dr+ r^2\, d\theta\otimes d\theta).
\]
Show that in these coordinates, $f$ is given by $r\mapsto \frac{r^2-1}{r^2+1}$ and 
\[
\nabla f= r\frac{\partial}{\partial r}. 
\]
Thus the flow line through $(1,\theta)$ is $t\mapsto (e^{-t}, \theta)$. 
\ehw
Here are a few simple facts about gradient flow lines that follow from general facts about flows. 
\bhw
Show that if $\gamma\co \R\to M$ is a gradient flow line, then $\gamma_a$ define by $\gamma_a(t)=\gamma(t+a)$ is also a gradient flow line (through the point $\gamma_a(0)=\gamma(a)$). 
\ehw
\bhw
If $\gamma$ and $\eta$ are gradient flow lines with the same image (this happens if and only if they have one point in common), then there is some $a$ such that $\gamma(t)=\eta(t+a)$. 
\ehw
\subsection{Spaces of gradient flow lines}
We will now define the Morse homology of a manifold. We begin by discussing the ``moduli space of gradient flow lines". Given a Morse function $f\co M\to \R$ and two critical points $p$ and $q$ of $f$ we define the space 
\begin{align*}
\mathcal{M}^p_q=\{ \gamma\co \R&\to M \text{ such that}\\
&\text{1)  $\gamma$ is a flow line of $-\nabla f$},\\
&\text{2) $\lim_{t\to -\infty} \gamma(t)=p$}, \text{ and}\\
&\text{3) $\lim_{t\to \infty}\gamma(t)=q$}
\}/\R
\end{align*}
where $\R$ acts on the gradient flow lines as indicated at the end of the last section. We can think of 
\[
\mathcal{M}^p_q = \text{ the space of unparameterized gradient flow lines from $p$ to $q$}.
\]
\bex
We will consider the function and metric in Example~\ref{ex:sphere} above. It is clear that there are only two critical points, $N$ and $S$. The only non-empty moduli space is 
\[
\mathcal{M}^N_S\cong S^1.
\]
To see this, notice that any gradient flow line intersects the equator of the sphere exactly once, and for every point on the equator, there is exactly one (unparameterized) flow line through that point. 
\eex
\bex\label{skewtorus}
We now consider a Morse function on the torus. One standardly draws the torus as, for example, the result of taking the circle of radius $1$ about $(0,2)$ in the $xz$-plane and rotating it about the $x$-axis, see Figure~\ref{fig:nonstorus}. If we restrict the projection map $\R^3\to\R\co (x,y,z)\mapsto z$ to this torus we will not get a ``generic" Morse function as discussed below, so in this example we will slightly skew this torus and let $f$ the restriction of the projection map to this torus. See Figure~\ref{fig:storus}.
\begin{figure}[htb]{
\begin{overpic}
{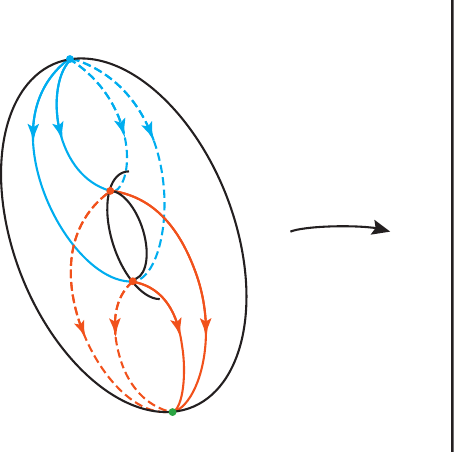}
\put(30, 195){$d$}
\put(46, 134){$c$}
\put(66, 68){$b$}
\put(80, 10){$a$}
\put(158, 118){$f$}
\put(225, 180){$z$-axis}
\end{overpic}}
\caption{The gradient flow of projection the the $z$-axis for the skew torus in $\R^3$.}
\label{fig:storus}
\end{figure}
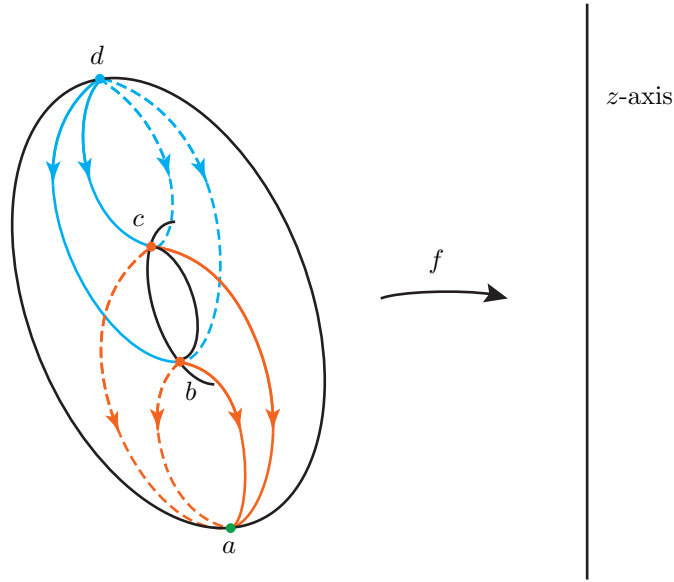  
We note there are $4$ critical points. A maximum, and hence index $2$ critical point, at $d$, a minimum, and hence index $0$ critical point, at $a$, and two index $1$ critical points at $b$ and $c$.  A few gradient flow lines for $f$ are shown in the figure. As it is difficult to draw more flow lines in this figure, we represent the torus as a square with opposite edges identified in Figure~\ref{fig:ftorus}. 
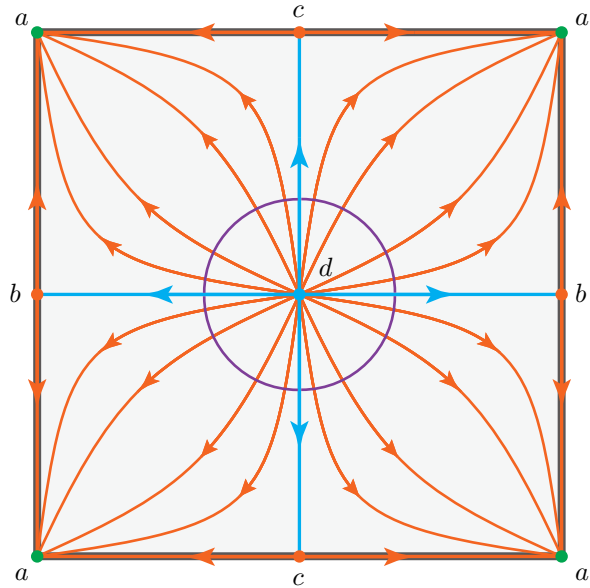
\begin{figure}[htb]{
\begin{overpic}
{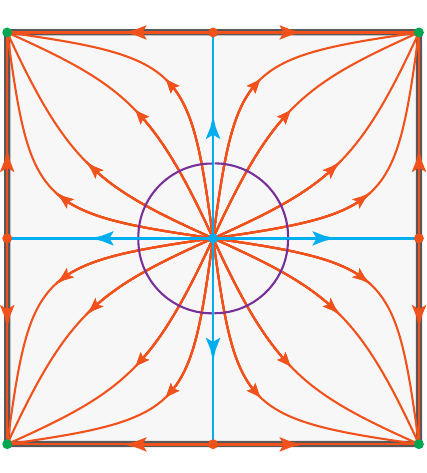}
\put(110, 120){$d$}
\put(100, 3){$c$}
\put(100, 218){$c$}
\put(-7, 110){$b$}
\put(207, 110){$b$}
\put(-5, 5){$a$}
\put(-5, 215){$a$}
\put(207, 5){$a$}
\put(207, 215){$a$}
\end{overpic}}
\caption{The torus represented as the square with opposite edges identified by translation. We see the gradient flow of the function $f$. The points in the purple curves are in one-to-one correspondence with flow lines from $d$ to $a$.}
\label{fig:ftorus}
\end{figure} 
There, we can draw more of the flow lines. We note that $\mathcal{M}^d_b, \mathcal{M}^d_c, \mathcal{M}^b_a,$ and $\mathcal{M}^c_a$ all consist of $2$ points, the first two spaces correspond to the blue lines and the second two correspond to the horizontal and, respectively, vertical orange lines. The spaces $\mathcal{M}^b_c$ and $\mathcal{M}^c_b$ are empty, and $\mathcal{M}^d_a$ is the union of $4$ open intervals, which are indicated by the purple curves in the figure. 
\eex
\begin{tcolorbox}[title={Morse Theorem 1 (Transversality)}]
For a generic choice of metric $g$ on $M$ and Morse function $f\co M\to \R$ the space 
\[
\mathcal{M}^p_q
\]
is a manifold of dimension $\idx p- \idx q-1$.
\end{tcolorbox}
The next example shows that the above theorem is not true for any metric and Morse function. 
\bex
In Figure~\ref{fig:nonstorus} we see the symmetric torus in $\R^3$. If $f$ is the restriction of the $z$-coordinate to the torus then we see the critical points and some of the gradient flow lines in the figure. 
\begin{figure}[htb]{
\begin{overpic}
{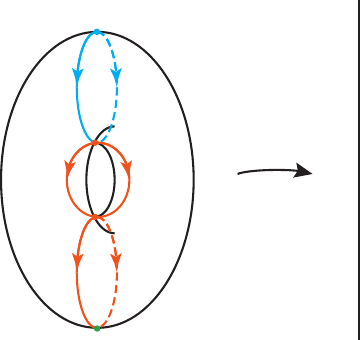}
\put(44, 152){$d$}
\put(44, 102){$c$}
\put(44, 45){$b$}
\put(44, -3){$a$}
\put(128, 88){$f$}
\put(180, 130){$z$-axis}
\end{overpic}}
\caption{The symmetric torus in $\R^3$.}
\label{fig:nonstorus}
\end{figure}
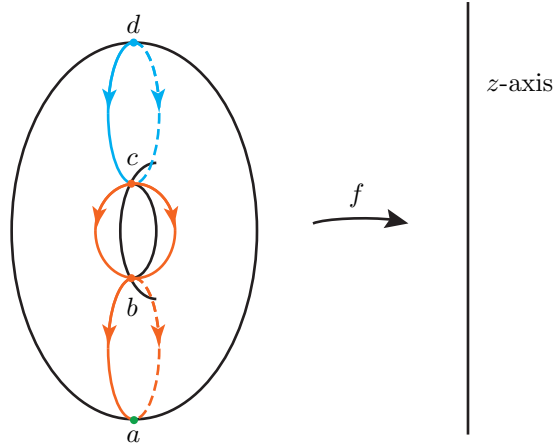  
Note that $\mathcal{M}^c_b$ consists of two points, but if we had chosen a generic Morse function and metric (here the metric is induced from $\R^3$) in the sense of Morse Theorem~1, then $\mathcal{M}^c_b$ would have dimension $-1$ and hence be empty. Thus it is clear that Morse Theorem~1 is not true for all Morse functions and metrics. 
\eex

We now move to the second main theorem in the case of Morse homology.
\begin{tcolorbox}[title={Morse Theorem 2 (Orientability)}]
If $M$ is oriented and we choose orientations on all the negative eigenspaces of $(T_pM, Hf_p)$ for all critical points $p$, then the $\mathcal{M}^p_q$ are ``consistently" oriented for all $p$ and $q$. 
\end{tcolorbox}
We will clarify what we mean by ``consistently" below when discussing ``gluing'', but now turn to the ``compactness" theorem (or, as mentioned above, it can really be thought of as the ``lack of compactness" theorem). 
\begin{tcolorbox}[title={Morse Theorem 3 (Compactness)}]
If $\{\gamma_i\}$ is a Cauchy sequence in $\mathcal{M}_q^p$ that does not converge to a point in $\mathcal{M}_q^p$, then $\{\gamma_i\}$ converges to a ``broken flow line". 
\end{tcolorbox}
We say that $\{\gamma_i\}$ converges to a broken flow line if there are critical points 
\[
p=r_1, r_2, \ldots, r_l=q
\] 
such that $\idx r_j>\idx r_{j+1}$ and there are elements $\widetilde\gamma^j\in \mathcal{M}^{r_j}_{r_{j+1}}$ such that for all $\epsilon>0$ there is some $N$ such that for all $i>N$ we have 
\[
\img(\gamma_i) \subset \text{$\epsilon$-neighborhood of } \left( \bigcup_{j=1}^{l-1} \img(\widetilde{\gamma_j})\right)
\]
and 
\[
 \left( \bigcup_{j=1}^{l-1} \img(\widetilde{\gamma_j})\right) \subset  \text{$\epsilon$-neighborhood of } \img(\gamma_i). 
\]
\bex
In Figure~\ref{fig:bfl} we see a sequence of gradient flow lines $\gamma_i$ shown as orange curves converging to a broken flow line given as the union of $\widetilde\gamma_1$, and $\widetilde\gamma_2$
\begin{figure}[htb]{
\begin{overpic}
{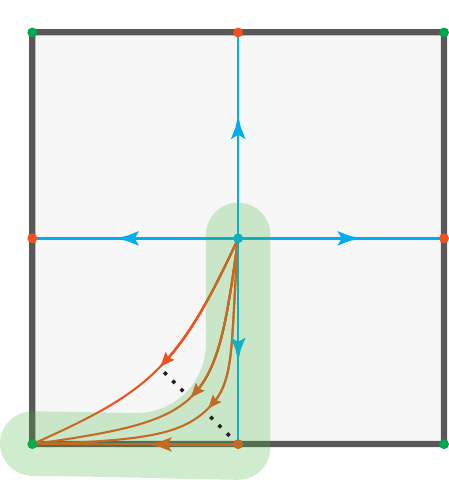}
\put(105, 120){$d$}
\put(112, 5){$c$}
\put(112, 222){$c$}
\put(3, 114){$b$}
\put(220, 114){$b$}
\put(4, 10){$a$}
\put(4, 218){$a$}
\put(220, 5){$a$}
\put(220, 218){$a$}
\put(65, 55){$\gamma_1$}
\put(94, 23){$\gamma_i$}
\put(120, 60){$\widetilde\gamma_1$}
\put(60, 5){$\widetilde\gamma_2$}
\end{overpic}}
\caption{A sequence of flow lines converging to a broken flow line.}
\label{fig:bfl}
\end{figure}
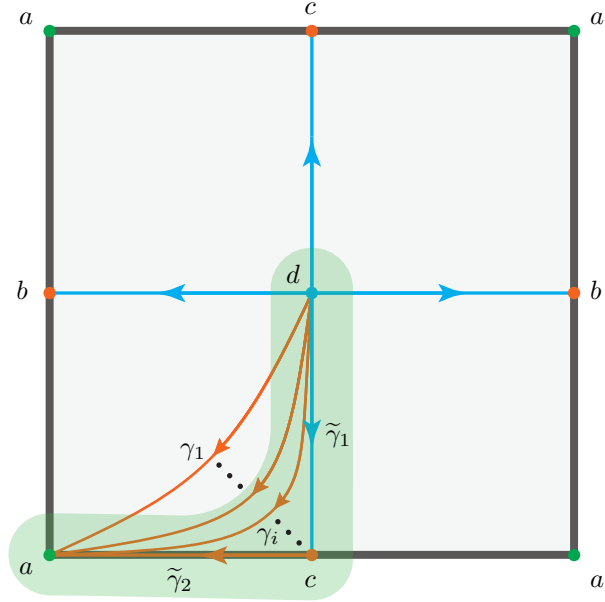 
The light green region is a neighborhood of the union of the images of $\widetilde\gamma_1$, and $\widetilde\gamma_2$, and we see that for large enough $i$, the images of the $\gamma_i$ are contained in the neighborhood. A similar figure will show that for large $i$ the images of $\widetilde\gamma_1$, and $\widetilde\gamma_2$ will be in a small neighborhood of the images of the $\gamma_i$. 
\eex
\begin{corollary}
If $\idx p=\idx q+1$, then $\mathcal{M}^p_q$ is a compact $0$-manifold. 
\end{corollary}
\begin{proof}
There are no critical points with index strictly between $\idx p$ and $\idx q$, so any Cauchy sequence in $\mathcal{M}^p_q$ must converge to a point in $\mathcal{M}^p_q$. 
\end{proof}
\begin{corollary}
If $\idx p= \idx q+2$, then 
\[
\overline{\mathcal{M}^p_q}\subset\left(\mathcal{M}_q^p 
\cup \left(
\bigcup_{\idx r=\idx p -1} \left(
\mathcal{M}^p_r\times \mathcal{M}^r_q
\right)
\right)
\right),
\]
where $\overline{\mathcal{M}^p_q}$ is the ``completion" formed by adding all broken trajectories. 
\end{corollary}
\bhw
Prove this corollary. 
\ehw

We finally turn to the final main theorem for Morse homology. 
\begin{tcolorbox}[title={Morse Theorem 4 (Gluing)}]
Let $p$, $r$, and $q$ be critical points with 
\[
\idx p=\idx r+1=\idx q+2.
\]
There is an orientation-preserving diffeomorphism 
\[
\Phi\co \left( (0,\epsilon)\times \mathcal{M}^p_r\times\mathcal{M}^r_q\right) \to \mathcal{M}^p_q
\]
such that for all $\gamma\in \mathcal{M}^p_r$, $\eta\in \mathcal{M}^r_q$, and sequence $\{s_i\}$ in $(0,\epsilon)$ converging to $0$ we have that $\Phi(\gamma,\eta,s_i)$ converges to the broken flow line $\gamma\cup \eta$. 
\end{tcolorbox}
\begin{remark}
Here is where we need the orientations from Morse Theorem 2 to be ``consistent". This means that we can relate the orientations on the moduli spaces as indicated in this theorem.
\end{remark}
\bex
We consider the compactification of the moduli spaces for the skew torus in Example~\ref{skewtorus} above. We note that $\mathcal{M}^d_a$ contains an open interval corresponding to the pink curves in Figures~\ref{fig:cpt}.
\begin{figure}[htb]{
\begin{overpic}
{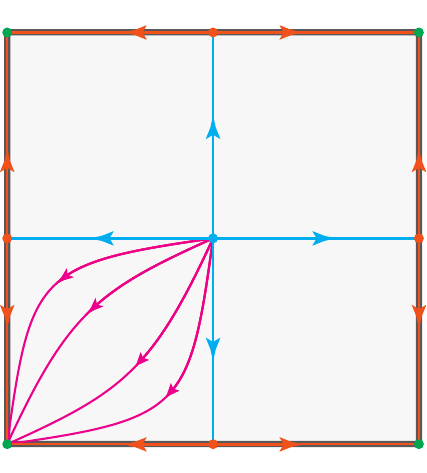}
\put(110, 120){$d$}
\put(100, 3){$c$}
\put(100, 218){$c$}
\put(-7, 110){$b$}
\put(207, 110){$b$}
\put(-5, 5){$a$}
\put(-5, 215){$a$}
\put(207, 5){$a$}
\put(207, 215){$a$}
\put(110, 60){$\gamma_1$}
\put(110, 162){$\gamma_2$}
\put(47, 122){$\gamma_3$}
\put(150, 122){$\gamma_4$}
\put(210, 162){$\delta_1$}
\put(210, 60){$\delta_2$}
\put(150, 1){$\delta_3$}
\put(47, 1){$\delta_4$}
\end{overpic}}
\caption{The compactification of the gradient flow lines in the lower left corner are the broken flow lines $\gamma_1\cup \delta_4$ and $\gamma_3\cup \delta_2$.}
\label{fig:cpt}
\end{figure}
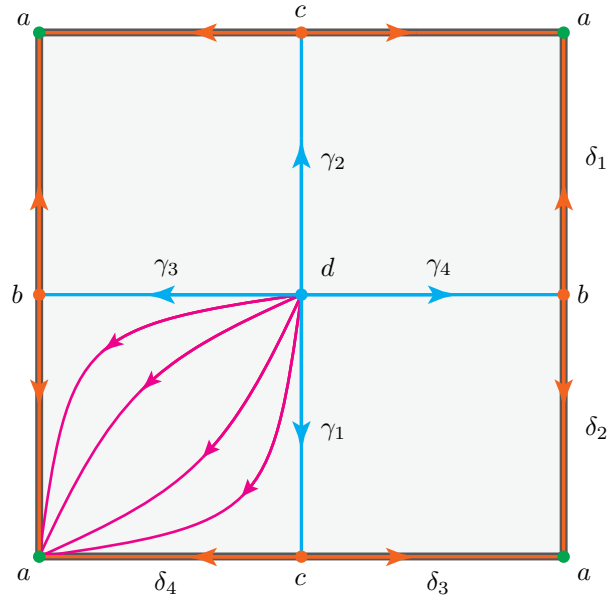 
The compactification of this interval consists of the broken flow lines $\gamma_1\cup \delta_4$ and $\gamma_3\cup \delta_2$. More generally we see that 
\begin{align*}
\mathcal{M}^d_b&=\{\gamma_1,\gamma_2\}, \quad \mathcal{M}^d_c=\{ \gamma_3, \gamma_4\},\\
\mathcal{M}^b_a&=\{\delta_3,\delta_4\}, \text{ and }\quad \mathcal{M}^c_a=\{ \delta_1,\delta_2\},
\end{align*}
and 
\[
\partial \overline{\mathcal{M}^d_a}= (\mathcal{M}^d_b\times\mathcal{M}^b_a)\cup (\mathcal{M}^d_c\times \mathcal{M}^c_a).
\]
So we saw that $\mathcal{M}^d_a$ consists of $4$ open intervals and thus its closure consists of $4$ closed intervals with boundary being $8$ broken flow lines. 
\eex
This example indicates that the proof of the following result. 
\begin{corollary}
If $k=\idx p=\idx q+2$, 
then $\overline{\mathcal{M}^p_q}$ is a compact $1$-manifold with boundary 
\[
\bigcup_{r \text{ of index } k -1} \left(
\mathcal{M}^p_r\times \mathcal{M}^r_q
\right).
\]
\end{corollary}
\bhw
Prove this corollary. 
\ehw

\subsection{Constructing Morse homology}

Let $f\co M\to \R$ be a Morse function and $g$ a metric on $M$ so that all the theorems in the past section hold. We define the chain groups for Morse homology to be
\[
C_i = \text{free abelian group generated by critical points of $\idx$ $i$}.
\]
We then define 
\[
\partial_i\co C_i\to C_{i-1}
\]
on the generators $p\in C_i$ by 
\[
\partial_i p= \sum_{q \text{ of index } i-1} |\mathcal{M}^p_q| \,\, q
\]
where $|\mathcal{M}^p_q|$ is the signed count of points in $\mathcal{M}^p_q$.

Notice that Morse Theorem~1 and~3 above show that $\mathcal{M}^p_q$ is a compact $0$-manifold, and Morse Theorem~2 shows that $\mathcal{M}^p_q$ is oriented. So $\partial_i q$ is well-defined. We note that if we considered $C_i$ to be a vector space over $\Z_2$ then we could just take the mod $2$ count of points in $\mathcal{M}^p_q$, and then we would not need Morse Theorem~2 (or the part of Morse Theorem~4 concerning orientations). 
\begin{theorem}\label{d2}
With the notation above we have 
\[
\partial_{i-1}\circ \partial_i=0.
\]
\end{theorem}
\begin{proof}
We note that 
\begin{align*}
\partial(\partial p)&=\partial \left(\sum_{r \text{ of index } i-1} |\mathcal{M}^p_r| r\right)\\
&=\sum_{q \text{ of index } i-2} |\mathcal{M}^p_r|  \left( \sum_{r \text{ of index } i-1} |\mathcal{M}^r_q| q\right)\\
&=\sum_{q \text{ of index } i-2} \left( \sum_{r \text{ of index } i-1} |\mathcal{M}^p_r| |\mathcal{M}^r_q|\right) q.
\end{align*}
So the coefficient on $q$ is $\sum_{r \text{ of index } i-1} |\mathcal{M}^p_r| |\mathcal{M}^r_q|$, but note this is exactly 
\[
\left|\bigcup_{r\in R} \left(
\mathcal{M}^p_r\times \mathcal{M}^r_q
\right)\right| = |\partial \overline{\mathcal{M}^p_r}|.
\]
Since $\overline{\mathcal{M}^p_r}$ is a compact $1$-manifold, we know the signed count of its boundary components is $0$. Thus $\partial\circ \partial =0$.
\end{proof}
We now give a ``picture proof" of this result. This is the version most people would give!
\begin{proof}[Picture Proof]
We begin by noting on the left-hand side of Figure~\ref{fig:picturepf} we see how $q$ would show up in $\partial^2 p$. That is, it will show up as a broken flow line. 
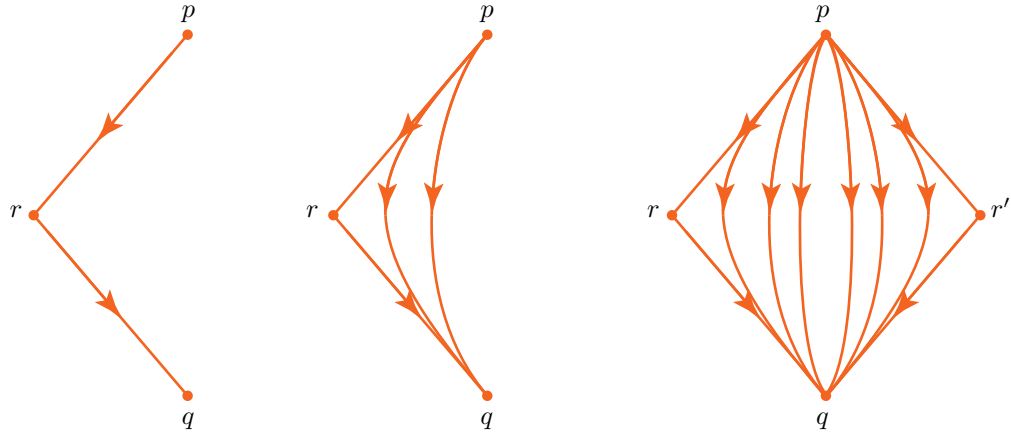
\begin{figure}[htb]{
\begin{overpic}
{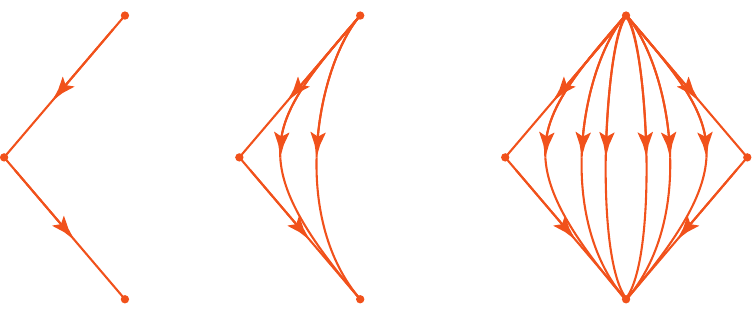}
\put(58, 151){$p$}
\put(-7, 75){$r$}
\put(58, -3){$q$}
\put(171, 151){$p$}
\put(105, 75){$r$}
\put(171, -3){$q$}
\put(298, 151){$p$}
\put(234, 75){$r$}
\put(298, -3){$q$}
\put(364, 75){$r'$}
\end{overpic}}
\caption{The compactification of the gradient flow lines in the lower left corner are the broken flow lines $\gamma_1\cup \delta_4$ and $\gamma_3\cup \delta_2$.}
\label{fig:picturepf}
\end{figure} 
But Morse Theorem~4 says that this will be a boundary component of $\overline{\mathcal{M}^p_q}$ as indicated in the middle of the figure. This component of $\overline{\mathcal{M}^p_q}$ is an oriented $1$-dimensional manifold. So this component of $\overline{\mathcal{M}^p_q}$ will look like the figure on the right of Figure~\ref{fig:picturepf}. So you see if $q$ is in $\partial^2 p$ then it will show up again with the opposite sign. Some people will just draw the diagram on the right of Figure~\ref{fig:picturepf} and call that the proof that $\partial^2=0$, and once you have thought through it, it is a great proof!
\end{proof}

We can now define the \dfn{$i^\text{th}$ Morse homology of $M$ to be}
\[
H_i(M)= \frac{\ker\left(\partial\co C_i\to C_{i-1}\right)}{\img\left(\partial\co C_{i+1}\to C_i\right)}.
\]
\bhw
Show that $H_i(M)$ agrees with the $i^\text{th}$ homology of $M$ as is usually defined in algebraic topology. \\ Hint: This might be more challenging than some previous exercises. Maybe read Section~\ref{aa} first and think about cellular homology and how you might get a $CW$ structure on $M$ from a Morse function. 
\ehw
Of course, if you finish this exercise, we know that the Morse homology does not depend on the choice of Morse function or the Riemannian metric, but we now turn to a more direct proof of this fact. 

\subsection{Invariance of Morse homology}
Suppose that $(f_0,g_0)$ and $(f_1, g_1)$ are two pairs of Morse functions and Riemannian metrics so that all the theorems in the previous section hold. Each $(f_i, g_i)$ gives us a chain complex $(C^i,\partial^i)$. We want to build a chain map
\[
\Phi\co (C^0,\partial^0)\to (C^1,\partial^1)
\]
that will induce an isomorphism on homology. To this end, let 
\begin{align*}
\Gamma= \{(f_t,g_t)\co& t\in \R \text{ such that }\\
&f_t=f_0 \text{ and } g_t=g_0 \text{ for } t\leq 0\\
&f_t=f_1 \text{ and } g_t=g_1 \text{ for } t\geq 1\}.
\end{align*}
We note that the space of functions and the space of metrics are contractible, so such a $\Gamma$ exists. Now let $V_t$ be the gradient $\nabla_{g_t} f_t$ of $f_t$ with respect to the metric $g_t$ and set $V$ to be the time-dependent vector field such that at time $t$ the vector field $V$ is $-V_t$:
\[
V=\{-V_t\}.
\]
A flow line of $V$ is a map $\gamma\co \R\to M$ such that $\gamma'(t)=(V_t)_{\gamma(t)}$. 
Now given a critical point $p$ of $f_0$ and a critical point $q$ of $f_1$ we set 
\begin{align*}
\widetilde{\mathcal{M}}^p_q=\{ \gamma\co \R&\to M \text{ such that}\\
&\text{1)  $\gamma$ is a flow line of $V$},\\
&\text{2) $\lim_{t\to -\infty} \gamma(t)=p$}, \text{and}\\
&\text{3) $\lim_{t\to \infty}\gamma(t)=q$}
\}.
\end{align*}
We note that since $V$ is a time-dependent vector field, there is no $\R$ action on $\widetilde{\mathcal{M}^p_q}$ as there was above. We now have analogs of Morse Theorems~1 through~4.

\begin{tcolorbox}[title={Morse Theorem 1' (Transversality)}]
For a generic choice of $\Gamma$ the space 
\[
\widetilde{\mathcal{M}}^p_q
\]
is a manifold of dimension $\idx p- \idx q$.
\end{tcolorbox}
We note that there is no $-1$ as in Morse Theorem~1 since there is no $\R$ action on flow lines of the time-dependent vector field $V$. 
\begin{tcolorbox}[title={Morse Theorem 2' (Orientability)}]
If $M$ is oriented and we choose orientations on all the negative eigenspaces of the critical points of $f_0$ and $f_1$, then the $\widetilde{\mathcal{M}}^p_q$ are ``consistently" oriented for all $p$ and $q$. 
\end{tcolorbox}
The meaning of ``consistently" is similar to its meaning above. In particular, it means the orientations in the gluing theorem below are consistent.
We now move to the compactness theorem.
\begin{tcolorbox}[title={Morse Theorem 3' (Compactness)}]
If $\{\gamma_i\}$ is a Cauchy sequence in $\widetilde{\mathcal{M}}_q^p$ that does not converge to a point in $\widetilde{\mathcal{M}}_q^p$, then $\{\gamma_i\}$ converges to ``broken flow line". 
\end{tcolorbox}
The notion of convergence is similar to the one above, but slightly different. 
We say that $\{\gamma_i\}$ converges to a broken flow line if there are critical points $p=p_1, p_2, \ldots, p_k$ of $f_0$ and critical points $p_{k+1}, \ldots, p_l=q$ of $f_1$ such that the indices are strictly decreasing except for possibly $p_k$ and $p_{k+1}$ which can have the same index, and elements $\widetilde{\gamma}_j\in\mathcal{M}(f_0)^{p_j}_{p_{j+1}}$ for $j<k$, $\widetilde{\gamma}_j\in\mathcal{M}(f_1)^{p_j}_{p_{j+1}}$ for $j>k$, and $\widetilde{\gamma}_k\in \widetilde{\mathcal{M}}^{p_k}_{p_{k+1}}$ such that the $\gamma_i$ converge to $\{\widetilde{\gamma}_j\}_{j=1}^l$ as in Morse Theorem~3.

We now turn to the ``gluing" theorem, but state it a bit vaguely and leave it to the reader to see how to adopt the statement of Morse Theorem~4 to the current situation. 
\begin{tcolorbox}[title={Morse Theorem 4' (Gluing)}]
Can ``undo" breaking with index difference $1$. 
\end{tcolorbox}
We now have the immediate corollaries whose proofs are very similar to the analogous ones in the previous section.
\begin{corollary}
If $p$ is a critical point of $f_0$, $q$ is a critical point of $f_1$ such that $\idx p = \idx q$, then $\widetilde{\mathcal{M}}^p_q$ is a compact $0$-manifold. 
\end{corollary}

\begin{corollary}
If $p$ is a critical point of $f_0$, $q$ is a critical point of $f_1$ such that $k=\idx p = \idx q+1$, then $\overline{\widetilde{\mathcal{M}}}^p_q$ is a compact $1$-manifold with boundary
\[
\left(\bigcup_{r\in \text{Crit}_{k-1}(f_0)} \left(
\mathcal{M}^p_r\times \widetilde{\mathcal{M}}^r_q
\right)\right)\bigcup\left(\bigcup_{r\in \text{Crit}_k(f_1)} \left(
\widetilde{\mathcal{M}}^p_r\times \mathcal{M}^r_q
\right)\right)
\]
where $\text{Crit}_k(f)$ is the set of critical points of $f$ with index $k$.  
\end{corollary}

We can now define $\Phi\co C_n^0\to C_n^1$ on a generator $p\in C^0_n$ by 
\[
\Phi_\Gamma(p)=\sum_{q\in \text{Crit}_n(f_1)} |\widetilde{\mathcal{M}}^p_q| q.
\]
Where, again, $|\widetilde{\mathcal{M}}^p_q|$ means the signed count of points in $\widetilde{\mathcal{M}}^p_q$. Clearly, Morse Theorems~1' through~3' show $\Phi_\Gamma$ is a well-defined homomorphism. We now see that this is a chain map. 
\begin{theorem}\label{tm26}
With the notation above
\[
\Phi_\Gamma \circ \partial_0 + \partial_1\circ \Phi_\Gamma=0.
\]
\end{theorem}
From basic homological algebra, we have the immediate corollary.
\begin{corollary}
The map $\Phi_\Gamma$ induces a homomorphism
\[
H(C^0,\partial^0)\to H(C^1,\partial^1),
\]
where $H(C^i,\partial^i)$ is the homology of $(C^i,\partial^i)$.
\end{corollary}
\begin{proof}[Proof of Theorem~\ref{tm26}]
We will denote by $\text{Crit}_k(f_i)$ the critical points of $f_i$ with index $k$. Now take a critical point $p$ of $f_0$ with index $n$. Then we have
\begin{align*}
(\Phi_\Gamma\circ\partial_0)(p)&+(\partial_1\circ\Phi_\Gamma)(p)
=\Phi_\Gamma\left(\sum_{r\in \text{Crit}_{n-1}(f_0)}|\mathcal{M}^p_r|r\right)
	+ \partial_1\left(\sum_{r\in\text{Crit}_n(f_1)}|\widetilde{\mathcal{M}}^p_r|r\right)\\
	&=\sum_{r\in\text{Crit}_{n-1}(f_0)}\left( |\mathcal{M}^p_r| \sum_{q\in\text{Crit}_{n-1}(f_1)} |\widetilde{\mathcal{M}}^r_q|q\right)
		+\sum_{r\in\text{Crit}_n(f_1)}\left(|\widetilde{\mathcal{M}}^p_r|\sum_{q\in\text{Crit}_{n-1}(f_1)}|\mathcal{M}^r_q|q\right)\\
	&=\sum_{q\in \text{Crit}_{n-1}(f_0)} \left(
	\sum_{r\in\text{Crit}_{n-1}(f_0)} |\mathcal{M}^p_r||\widetilde{\mathcal{M}}^r_q| +
		\sum_{r\in\text{Crit}_n(f_1)}|\widetilde{\mathcal{M}}^p_r||\mathcal{M}^r_q| \right) q\\
	&=\sum_{q\in \text{Crit}_{n-1}(f_0)}|\partial \overline{\widetilde{\mathcal{M}}^p_q}| q= 0,
	\end{align*}
since $\overline{\widetilde{\mathcal{M}}^p_q}$ is a compact $1$-manifold with boundary. 
\end{proof}
In Figure~\ref{fig:cm} we see the ``picture proof" of this theorem. Here $p\in\text{Crit}_n(f_0), q\in\text{Crit}_{n-1}(f_1)$ and $r$ and $r'$ are either in $\text{Crit}_n(f_1)$ or $\text{Crit}_{n-1}(f_0)$.
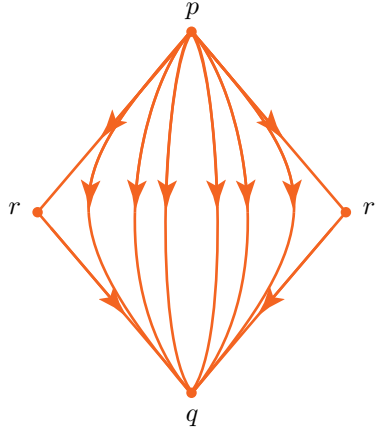
\begin{figure}[htb]{
\begin{overpic}
{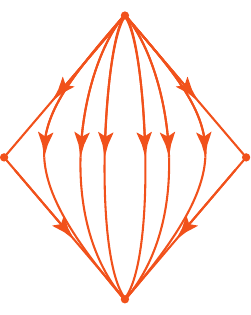}
\put(58, 151){$p$}
\put(-9, 75){$r$}
\put(58, -3){$q$}
\put(125, 75){$r'$}
\end{overpic}}
\caption{A component of $\overline{\widetilde{\mathcal{M}}^p_q}$.}
\label{fig:cm}
\end{figure} 
So we see the terms showing up in $(\Phi_\Gamma \circ \partial_0 + \partial_1\circ \Phi_\Gamma)(p)$ are simply the boundary of $\overline{\widetilde{\mathcal{M}}^p_q}$.

We would now like to see that $\Phi_\Gamma$ induces an isomorphism. To achieve this we consider two paths $\Gamma_0$ and $\Gamma_1$ from $(f_0,g_0)$ to $(f_1,g_1)$ and let $D$ be a path from $\Gamma_0$ to $\Gamma_1$. That is,
\[
D=\{(f_{s,t}, g_{s,t}) \text{ for } (s,t)\in [0,1]\times \R \text{ and } \Gamma_0=(f_{0,t},g_{0,t}), \Gamma_1=(f_{1,t}, g_{1,t})\}.
\]
Now if $p\in\text{Crit}(f_0)$ and $q\in\text{Crit}(f_1)$ then we set 
\[
\widehat{\mathcal{M}}^p_q=\{(\gamma,s): \gamma\in\widetilde{\mathcal{M}}^p_q(s)\},
\]
where $\widetilde{\mathcal{M}}^p_q(s)$ is $\widetilde{\mathcal{M}}^p_q$ for the path $\Gamma_s=(f_{s,t},g_{s,t})$. We have the ``same" Morse Theorems~1 through~4 for $\widehat{\mathcal{M}}^p_q$ except that 
\[
\dim \widehat{\mathcal{M}}^p_q= \idx p-\idx q+1
\]
and the ``obvious'' changes to the compactness and gluing theorems.
\bhw
Figure out what the ``obvious" changes are. 
\ehw
Now we can define $K\co C_n^0\to C_{n+1}^1$ by
\[
K(p)=\sum_{q\in\text{Crit}_{n+1}(f_1)} |\widehat{\mathcal{M}}^p_q| q.
\]
This gives us a chain homotopy equivalence between $\Phi_{\Gamma_0}$ and $\Phi_{\gamma_1}$.
\begin{theorem}
With the notation above
\[
\partial_1\circ K + K\circ \partial_0 + \Phi_{\Gamma_0}+\Phi_{\Gamma_1}=0.
\]
\end{theorem}
\bhw
Using the previous exercise, prove this theorem. \\ Hint: Just as in the previous theorems, notice that the count of points in $\partial  \overline{\widehat{\mathcal{M}}^p_q}$ is exactly the coefficient of $q$ in the application of the left-hand side of the formula in the theorem to $p$. 
\ehw
Of course, standard homological algebra tells us that the previous theorem shows $\Phi_{\Gamma_0}$ and $\Phi_{\Gamma_1}$ induce the same map on homology (or the way it is stated, one map is minus the other).
\begin{corollary}
With the notation above the maps $\Phi_{\Gamma_0}$ and $-\Phi_{\Gamma_1}$ induce the same maps on homology. 
\end{corollary}
To complete our proof that the Morse homology is independent of the choices of Morse function and Riemannian metric, we have the following exercises. 
\bhw
Let $\Gamma$ be the constant path from $(f,g)$ to $(f,g)$. Show that $\Phi_\Gamma$ induces the identity map on homology. \\
Hint: Of course, one must perturb $\Gamma$ so that it is ``generic" and Theorems~$1'$ through~$4'$ hold, but this perturbation can be chosen to be arbitrarily small. 
\ehw
\bhw
If $\Gamma_1$ is a path from $(f_0,g_0)$ to $(f_1,g_1)$ and $\Gamma_2$ is a path from $(f_1,g_1)$ to $(f_2,g_2)$, then we can define the concatenated path $\Gamma$ that agrees with $\Gamma_1$ for $t\leq 1$ and agrees with $\Gamma_2$, shifted by $2$, for $t\geq 1$. Show that $\Phi_{\Gamma_2}\circ \Phi_{\Gamma_1}=\Phi_{\Gamma}$ on homology. 
\ehw
\begin{theorem}
The Morse homology of $M$ is independent of the Morse function $f$ and Riemannian metric $g$.
\end{theorem}
\bhw
Use the above exercises and theorems to prove this. 
\ehw

\subsection{A few details of the main theorems}\label{details}
In this section, we will give an indication of how to prove Morse Theorems~$1$ through~$4$ (and hence $1'$ through $4'$ too). We will not give complete proofs as they are quite intricate and involve a lot of analysis that can obscure the main ideas on a first pass at the subject. One may find more details in \cite{FrauenfelderNicholls2020Pre, HutchingsNotes, Schwarz1993}; in particular, the presenation here was partially inspired by the second listed reference. We also note that one can define Morse homology with chain groups being $\Z/2\Z$ vector spaces generated by critical points, and in this case, one does not need Morse Theorem~$2$ about the orientability of the spaces $\mathcal{M}^p_q$. In many situations, it is common for people to work over $\Z/2\Z$ and ignore this theorem. We will do that here and only discuss the other theorems. 

We will not describe the ideas of the proof of the main theorems for the other homology theories, but they mostly follow the same general idea (though details are quite different), except for the proof of the compactness result, which are quite different. 

\subsubsection{Morse Theorem~$1$}
Given a Morse function $f\co M\to \R$ and a Riemannian metric $g$ on $M$ we define the \dfn{path space} between two points $p$ and $q$ to be
\[
\mathcal{P}^p_q=\{ \gamma\co \R\to M : \lim_{t\to -\infty}\gamma(t)=p \text{ and } \lim_{t\to\infty}\gamma(t)=q\}.
\]
We will need to know what the tangent space of $\mathcal{P}^p_q$ is. (At this point, we are ignoring issues of smoothness. That is, we are not specifying if $\gamma$ is smooth or $C^k$ or something else. We will discuss this later.) To this end, recall that one may think of a tangent vector at a point as the equivalence class of paths through that point that have the same first derivative at that point. So if $\gamma\in\mathcal{P}^p_q$, then we can consider a path $s\to \gamma_s$ for $s\in(-\epsilon, \epsilon)$ with $\gamma_0=\gamma$. Notice that $\gamma_s(t_0)$ for a fixed $t_0$ is a path in $M$ such that $\gamma_0(t_0)$ is $\gamma(t_0)$. Thus 
\[
v(t_0)=\frac{\partial \gamma_s(t_0)}{\partial s}\Big|_{s=0}
\]
is a vector in $T_{\gamma(t_0)}M$. See Figure~\ref{fig:tangent}. 
\begin{figure}[htb]{
\begin{overpic}
{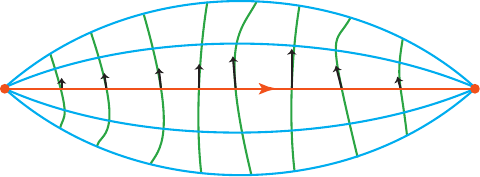}
\put(-9, 40){$p$}
\put(235, 40){$q$}
\put(125, 48){$\gamma$}
\end{overpic}}
\caption{The blue is a path $\gamma_s$ in $\mathcal{P}^p_q$ through $\gamma$ (shown in orange). The green curves are the paths $s\mapsto \gamma_s(t_0)$ for some fixed $t_0\in\R$ and the black vectors give the vector field $v$ along $\gamma$.}
\label{fig:tangent}
\end{figure} 
So $v(t)$ for $t\in\R$ is a vector field along $\gamma$, by this we mean that for each $\gamma(t)$ the vector $v(t)\in T_{\gamma(t)}M$. Summing up, given a path in $\mathcal{P}^p_q$ through $\gamma$ its ``derivative" at $\gamma(t)$ is a vector field $v(t)$ along $\gamma$. We also know that since $\gamma_s\in \mathcal{P}^p_q$ we must have that $v(t)\to 0$ as $t\to \pm\infty$. 

We can easily see the converse too. Given any vector field $v(t)$ along $\gamma$ that decays to $0$ as $t$ approaches $\pm\infty$, extend $v$ to a vector field $V$ on all of $M$. Let $\Gamma\co M\times\R\to M$ be the flow of $V$. We can now define
\[
\gamma_s(t)=\Gamma(\gamma(t),s).
\]
This gives a path in $\mathcal{P}^p_q$ through $\gamma$ and by the properties of the flow of a vector field, we see that the vector field associated to $\gamma_s$ along $\gamma$ is simply $v$. Thus, vectors tangent to $\mathcal{P}^p_q$ at $\gamma$ are in one-to-one correspondence with vector fields along $\gamma$ that decay to $0$ as $t$ approaches $\pm\infty$. So we see
\begin{align*}
T_\gamma(\mathcal{P}^p_q)&=\{\text{vector fields along $\gamma$ that decay to $0$ as $t$ approaches $\pm\infty$}\}\\
&=\Gamma_d(\gamma^*TM)
\end{align*}
where $\gamma^*TM$ is the pull-back of $TM$ to $\R$ by $\gamma$ and $\Gamma_d$ denotes sections of this bundle that decay as indicated. 

We now suppose that $p$ and $q$ are critical points of $f$ and consider the map
\[
S\co \mathcal{P}^p_q\to T\mathcal{P}^p_q\co \gamma\mapsto \gamma'(t)+ \nabla f_{\gamma(t)},
\]
where $\nabla f$ is the gradient of $f$ with respect to $g$. Clearly $\gamma'(t)+ \nabla f_{\gamma(t)}$ is an element of $\Gamma_d(\gamma^*TM)$ as desired and also that the space of {\em parameterized} flow lines, which we denote by $\underline{\mathcal{M}}^p_q$, is 
\[
\underline{\mathcal{M}}^p_q=S^{-1}(Z)
\]
where $Z$ is the zero section of $T\mathcal{P}^p_q$. We note that $\R$ acts on $\underline{\mathcal{M}}^p_q$ freely by reparameterization and thus if $\underline{\mathcal{M}}^p_q$ is a manifold then so is ${\mathcal{M}}^p_q$.
So we have managed to describe the space we are interested in, that is $\mathcal{M}^p_q$, as the preimage of a submanifold $Z$ under some function $S$ (after modding out by reparameterization). Thus, from standard differential topology, if $S$ were ``transverse" to $Z$ then we would expect $\underline{\mathcal{M}}^p_q$ to be a manifold. Of course, the problem here is that $\mathcal{P}^p_q$ and $T\mathcal{P}^p_q$ are infinite-dimensional, and a standard differential topology course does not deal with these. So we will take a brief digression to address this and then return to a sketch of the proof that $\mathcal{M}^p_q$ is a manifold. 
\subsubsection{Recollections on transversality and infinite dimensional manifolds}
We begin by recalling the ``standard way" to create functions transverse to submanifolds in finite dimensions. Given a smooth function $f\co X\to Y$ between manifolds and a submanifold $Z$ of $Y$ we would like to ``make" $f$ transverse to $Z$, by which we mean we would like to find arbitrarily small perturbations of $f$ so that it is transverse to $Z$ (if it is not already transverse). The idea is to try to choose a perturbation space $P$ and extend $f$ to a map $F\co X\times P\to Y$. Since $P$ adds dimensions to the domain, $F$ has a better chance of being transverse to $Z$ (and of course, for the right space $P$, it will be). We then use the following standard theorem from differential topology. 
\begin{theorem}\label{basictransversality}
With the notation in the previous paragraph, assume that $F$ is transverse to $Z$ and let $\pi\co X\times P\to P$ be the projection map. If $p$ is a regular value of $\pi|_{F^{-1}(Z)}\co F^{-1}(Z)\to P$, then $f_p(\cdot)=F(\cdot, p)\co X\to Y$ is transverse to $Z$. 
\end{theorem}
\bhw
Prove this theorem for $X, Y,$ and $P$ of arbitrary dimension. 
\ehw
In finite dimensions have we have Sard's Theorem, which tells us when there is a dense set of regular values $p$.
\begin{theorem}[Sard's Theorem]
If $h\co M\to N$ is a $C^k$-map then the set of regular values of $h$ is dense in $N$ if $k>\max\{0,m-n\}$. 
\end{theorem}
Thus returning to our original problem stated above, in the finite dimensional setting, if $F$ is smooth enough, then there will be $f_p\co X\to Y$ transverse to $Z$ for a dense set of $p$ and thus $f$ may be ``approximated" by a map transverse to $Z$ (once we find $F$!). 

The infinite-dimensional setting is a bit more complicated. For this, we begin with a few definitions. First, recall that a linear space $V$ with a norm $\| \cdot\|\co V\to \R$ is a \dfn{Banach space} if any Cauchy sequence in $V$ (that is sequence $v_i$ such that for any $\epsilon>0$ there is some $N$ such that $\| v_i-v_j\|< \epsilon$ for all $i, j>N$) converges to a point in $V$. Given two Banach spaces $V$ and $W$ (we will suppress the norms, though they are essential) we say a bounded linear map $L\co V\to W$ is a \dfn{Fredholm operator} if 
\begin{itemize}
\item $L$ has a closed image,
\item the dimension of $\ker F$ is finite, and
\item the dimension of $\rm{coker} F$ is finite,
\end{itemize}
where the $\rm{coker} F$ is simply $W/F(V)$ and is called the \dfn{co-kernel}. We define the index of a Fredholm operator to be
\[
\idx F= \dim \ker(F) - \dim\rm{coker}(F).
\]
Finally, a \dfn{Banach manifold} $M$ is a topological manifold with coordinate charts modeled on a Banach space $(V,\|\cdot\|)$ and having smooth transition maps. We can now generalize Sard's theorem.
\begin{theorem}[Sard-Smale]
Let $X$ and $Y$ be separable Banach manifolds and $h\co X\to Y$ be a $C^k$-map such that 
\[
dh_x\co T_xX\to T_{f(x)}Y
\]
is Fredholm of index $l$ for all $x\in X$. If $k\geq \max\{1, l+1\}$, then the regular values of $h$ are dense in $Y$.  (And for such a $y$, $h^{-1}(y)$ is a manifold of dimension $l$.)
\end{theorem}
With this theorem and an improved version of Theorem~\ref{basictransversality} we have the following result that will be useful to prove $\mathcal{M}^p_q$ is a manifold. We first set up some notation. Let $X$ and $P$ be Banach manifolds, and $\pi\co E\to X\times P$ be a Banach vector bundle (the reader should think through what this might mean). Suppose $\phi\co (X\times P)\to E$ is a section and $Z$ is the zero section of $E$. We want to discuss when $\phi$ is transverse to $Z$. Note that $\ker d\pi_e\co T_eE\to T_{\phi(e)}(X\times P)$ is naturally identified with the fiber $E_{\phi(e)}$ of $E$ above $\phi(e)$, since the tangent space to a linear space is naturally the linear space. So $T_eE= \ker d\pi_e\oplus H_x$ where $H_e$ is a ``horizontal" subspace of $T_eE$; that is, it is transverse to $\ker d\pi_e$ and its intersection with $\ker d\pi_e$ is $\{0\}$. In general $H_e$ is determined by a ``connection" on the bundle $E$ but if $e\in Z$ then we can take $H_e=T_e Z$, so there is no need for a connection. We can now define $\nabla_{\phi(e)}\phi$ to be the composition of $d\phi_{\phi(e)}$ with the projection to $E_{\phi(e)}$ using the splitting $T_e E= E_{\phi(e)}\oplus T_eZ$. 

\begin{theorem}\label{maintransverse}
With the notation above,
suppose $\phi\co (X\times P)\to E$ is a section of the bundle and for all $(x,p)\in \phi^{-1}(Z)$, we have 
\begin{enumerate}
\item $\nabla_{(x,p)}\phi \co T_{(x,p)}(X\times P)\to T_{\phi(x,p)}\pi^{-1}((x,p))$ is surjective, and 
\item if $\phi_p$ is the map $\phi_p\co X\to E\co x\mapsto \phi(x,p)$ then $\nabla_{(x,p)}\phi_p\co T_xX\to T_{\phi(x,p)}\pi^{-1}((x,p))$ is Fredholm of index $l$.
\end{enumerate}
Then, for a generic $p\in P$ we have $\phi_p^{-1}(Z)$ is a manifold of dimension $l$.
\end{theorem}
We can finally sketch the proof of Morse Theorem~$1$.
\begin{proof}[Sketch of the proof of Morse Theorem~$1$]
Recall we have the space $\mathcal{P}^p_q$ of paths between critical points $p$ and $q$ of a Morse function $f$, and its tangent space at $\gamma$ is $\Gamma_d(\gamma^*TM)$. Trying to apply the theorems above, we choose our space of perturbations to be the set $P$ of all Riemannian metrics on $M$. Now let $E$ be the bundle over $\mathcal{P}^p_q\times P$ with fiber at $(\gamma,g)$ the space $\Gamma_d(\gamma^*TM)$. We can now define the section 
\[
\phi\co (\mathcal{P}^p_q\times P)\to E\co (\gamma, g)\mapsto \gamma'(t)+V(\gamma(t))
\]
where $V=\nabla_g f$ is the gradient of $f$ with respect to $g$. 
\bhw
Show that $\phi(\gamma,g)$ is the zero vector field along $\gamma$ if and only if $\gamma$ is a gradient flow line of $\nabla_g f$ from $p$ to $q$. 
\ehw

Now, if we can find the appropriate function spaces for $\mathcal{P}^p_q$, $P$, and $E$, such that $\phi$ is Fredholm and transverse to the zero section in $E$, then there would be a dense set of $g$ for which the Moduli space $\underline{\mathcal{M}}^p_q$ is a manifold. Moreover, if we can compute that the index of $\phi$ is $\idx p-\idx q $ then it would also have the desired dimensions. 

For example, for $P$ one could consider the space of $C^k$-smooth Riemannian metrics on $M$, and for the other spaces one could choose appropriate Sobolev spaces where some number of derivatives have bounded $L^2$-integrals. To carry out the details would be a long diversion for these lecture notes, so we refer the reader to \cite{HutchingsNotes}, but we do briefly indicate why $\phi$ will be transverse to the zero section in the right functional analytic setup. 

Given a tangent vector $(\widetilde{\gamma},\widetilde{g})$ in $T_{(\gamma,g)}(\mathcal{P}^p_q\times P)$ we can represent it by a path $(\gamma_t, g_t)$ such that $(\gamma_0,g_0)=(\gamma,g)$ and $(\dot \gamma_t|_{t=0}, \dot g_t|_{t=0})=(\widetilde{\gamma},\widetilde{g})$ where the dot represents the derivative with respect to the $t$ parameter. Now
\[
\phi(\gamma_t, g_t)(s)= \gamma'_t(s)+ V^t_{\gamma_t(s)}
\]
where $V^t$ is the gradient of $f$ with respect to $g_t$, and the prime denotes the derivative with respect to $s$. Now
\begin{align*}
d\phi_{(\gamma,g)}(\widetilde{\gamma},\widetilde{g})&=\frac{d}{dt} \phi(\gamma_t,g_t)|_{t=0}\\
&= \frac{d}{dt}\left(\frac{d}{ds} \gamma_t(s)\right)\big|_{t=0} + \frac{d}{dt} V^0_{\gamma_t(s)}\big|_{t=0} + \frac{d}{dt} V^t_{\gamma_0(s)}\big|_{t=0}.
\end{align*}
If we choose our path so that $\gamma_t=\gamma$ for all $t$, then we see that 
\[
d\phi_{(\gamma,g)}(\widetilde{\gamma},\widetilde{g})=\frac{d}{dt} V^t_{\gamma(s)}\big|_{t=0}.
\]
We let $\dot V$ be the projection of $\frac{d}{dt} V^t_{\gamma(s)}\big|_{t=0}$ to the fiber of the bundle $E$ (that is to $\Gamma_d(\gamma^*TM)$).
To work out $\dot V$, we recall that $V^t$ is the gradient of $f$ with respect to $g_t$. If we choose local coordinates $x_1,\ldots, x_n$, then we can write
\[
g_t=\sum_{i,j=1}^n (g_t)_{ij} dx^i\otimes dx^j
\]
for some real valued functions $(g_t)_{ij}$. Thinking of $\begin{pmatrix} (g_t)_{ij}\end{pmatrix}$ as a matrix we let $\begin{pmatrix} (g_t)^{ij}\end{pmatrix}$ be its inverse and since 
\[
df= \sum_{i=1}^n \frac{\partial f}{\partial x_i} \, dx_i
\]
we see that 
\[
V^t=\sum_{i,j=1}^n 
(g_t)^{ij}\big|_{t=0} \frac{\partial f}{\partial x_i}\frac{d }{ dx_j}.
\]
Thus
\[
\dot V= \sum_{i,j=1}^n \frac{d}{dt}
(g_t)^{ij}\big|_{t=0} \frac{\partial f}{\partial x_i}\frac{d }{ dx_j}.
\]

We now suppose that $v$ is a vector field along $\gamma$ in $\Gamma_d(\gamma^*TM)$. In local coordiates we can write it $\sum_{i=1}^n  a_i\, \frac{d}{dx_i}$.  Let $A$ be the matrix 
\[
\begin{pmatrix}
a_1& a_2 & \cdots &a_n\\
a_2& 0&\cdots & 0\\
\vdots&& \ddots&\\
a_n&0& \cdots & 0
\end{pmatrix}
\]
with $v$ in the first column and first row and $ 0$'s elsewhere. Now if $g$ is given by the matrix $G=\begin{pmatrix} g_{ij}\end{pmatrix}$ then define the matrix $G_t^{-1}=G^{-1} + tA$ and $G_t$ as its inverse. For $t$ small, the inverse will exist and define a Riemannian metric $g_t$.  We can further choose coordinates near some point on $\gamma$ so that $\nabla f=\begin{pmatrix} 1 & 0 & \cdots & 0\end{pmatrix}^T$. Thus, we see that near this point on $\gamma$ the vector $\dot V=v$. Using a cutoff function, we can find a global family of metrics $g_t$ on $M$ so that near the desired point we still have $\dot V=v$. That is, locally any element of $\Gamma_d(\gamma^*TM)$ is in the image of $\nabla_{(\gamma,g)}\phi$. 

In the appropriate functional analytic setting, this implies that $\nabla_{(\gamma,g)}\phi$ is onto $E_{\gamma,g}=\Gamma_d(\gamma^*TM)$. We sketch some details here. To this end, we need to choose a Banach structure on the bundle $E$. Given a fixed metric on $M$ and we can define the $L^2$-innner product on elements $f,g\in \Gamma_d(\gamma^*TM)$ to be
\[
\langle f,g\rangle= \int_\R \langle f(s), g(s)\rangle\, ds,
\]
where the inner product in the integrand is coming from the metric on $M$. This gives a norm $\|f\|^2=\langle f, f\rangle$ and hence a metric $\Gamma_d(\gamma^*TM)$ thought of as smooth sections. We can now complete this space with respect to this metric to obtain $L^2_d(\gamma^*TM)$, the space of $L^2$-integrable sections of $\gamma^*TM$. 

We would now like to show that the orthogonal complement (with respect to the $L^2$-inner product) of the image of $d\phi_{(\gamma,g)}$ in $L^2_d(\gamma^*TM)$ is the trivial vector space $\{0\}$. So let $\omega$ be an element of $L^2_d(\gamma^*TM)$ orthogonal to the image of $d\phi_{(\gamma,g)}$. Recall that $d\phi_{(\gamma,g)}(0,\widetilde{g})=\dot V$, where $\dot V$ is defined above. We will show that anything orthogonal to $\dot V$ for all $\dot V$ is actually zero (notice that this shows that a vector even being orthogonal to a subset of $\img(d\phi_{(\gamma,g)})$ is enough to show it is zero). So we assume that, for any $\dot V$, we have
\[
\int_\gamma\langle \dot V,\omega\rangle\, ds = 0.
\]
Now, for any $x\in\R$, we can choose a neighborhood $N$ of $\gamma(x)$ so that $N\cap \img(\gamma)$ is the image of an interval $I$. Choose a function $h\co \R\to \R$ so that $h=1$ near $x$, $h=0$ oustide $I$, and $h\geq 0$. Now set $v=h\omega$. From above, we can choose $\widetilde{g}$ so that $\dot V=v$ where $h$ is non-zero. We now compute
\[
0=\int_\R \langle \dot V, \omega\rangle\, ds = \int_\R h\langle \omega, \omega\rangle\, ds.
\]
This implies that $h\|\omega\|^2=0$ and so $\|\omega\|=0$ near $x$. But since $x$ was arbitrary $\omega=0$, as we wanted to prove. Note, there are details here involving the Banach structure we use on the domain and the continuity (or lack thereof) of $\omega$, but we leave it to the reader and the references above to sort out these details. The main point was to give an indication of the flavor of the argument necessary to prove some sort of ``transversallity'' result.

In a similar fashion, but with more work, one may also show that $\nabla_{(x,p)}\phi$ for a fixed $g$ satisfies the conditions in Theorem~\ref{maintransverse}, which gives us Morse Theorem~~$1$. 
\end{proof}

\subsubsection{Morse Theorem~$3$} Given a Morse function $f\co M\to \R$ we choose a Riemannian metric $g$ so that Morse Theorem~$1$ holds. We want to consider a sequence of flow lines from the critical point $p$ to the critical point $q$. Let $\gamma_i$ be such a sequence and assume it is a Cauchy sequence, so it should converge. To analyze this situation, we note that one can choose a small ball $B_1$ about $p$ so that $B_1$ contains no other critical points and any gradient flow line from $p$ transversely intersects $\partial B_1$ in one point. 
\bhw
Use the Morse Lemma, Theorem~\ref{ML}, to prove that such a ball exists. 
\ehw
Let $x_i$ be the unique point at which $\gamma_i$ intersects $\partial B_1$. One may easily check that $\{x_i\}$ is a Cauchy sequence in the compact set $\partial B$ and so it converges to a point $x_\infty$. See the upper left diagram in Figure~\ref{fig:bfl2}. Now let $\widetilde\gamma_1$ be the gradient flow line through $x_\infty$. If $\widetilde\gamma_1$ flows to $q$ then the $\gamma_i$ converge to an element of $\mathcal{M}^p_q$. We now need to see that the $\gamma_i$ converge to a broken flow line. 
\begin{figure}[htb]{
\begin{overpic}
{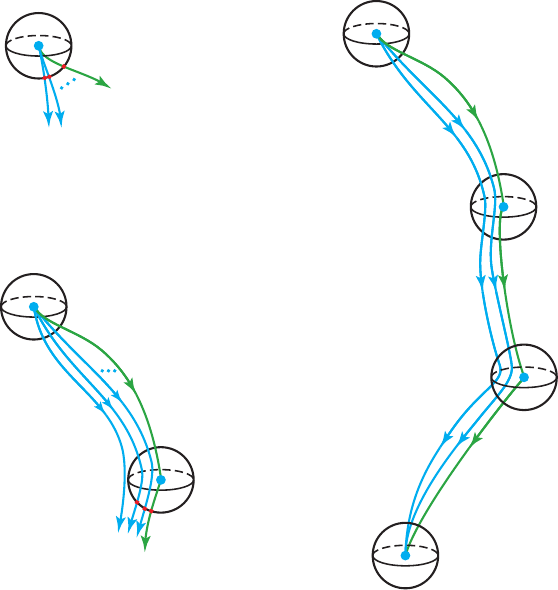}\footnotesize
\put(10, 261){\color{lblue}$p$}
\put(10, 225){\color{lblue}$\gamma_1$}
\put(33, 225){\color{lblue}$\gamma_2$}
\put(48, 249){\color{dgreen}$\widetilde\gamma_1$}
\put(12, 241){\color{red}$x_1$}
\put(33, 252){\color{red}$x_\infty$}
\put(8, 135){\color{lblue}$p$}
\put(82, 51){\color{lblue}$r_2$}
\put(20, 100){\color{lblue}$\gamma_1$}
\put(70, 100){\color{dgreen}$\widetilde\gamma_1$}
\put(73, 25){\color{dgreen}$\widetilde\gamma_2$}
\put(172, 267){\color{lblue}$p$}
\put(246, 183){\color{lblue}$r_2$}
\put(256, 101){\color{lblue}$r_3$}
\put(200, 15){\color{lblue}$q$}
\put(223, 243){\color{dgreen}$\widetilde\gamma_1$}
\put(250, 140){\color{dgreen}$\widetilde\gamma_2$}
\put(222, 50){\color{dgreen}$\widetilde\gamma_3$}
\end{overpic}}
\caption{Converging to a broken flow line. In the top left, we see the sequence of gradient flow lines in $\mathcal{M}^p_q$ near $p$. They intersect $\partial B_1$ in points $x_i$ that converge to $x_\infty$. The flow line $\widetilde\gamma_1$ is the flow line through $x_\infty$. On the bottom left we see the critical point $r_2$ to which $\widetilde{\gamma}_1$ flows. We also see the ball $B_2$ around $r_2$ used to find the next flow line in the broken flow line. On the right we see the complete broken flow line to which the $\gamma_i$ converge.}
\label{fig:bfl2}
\end{figure}
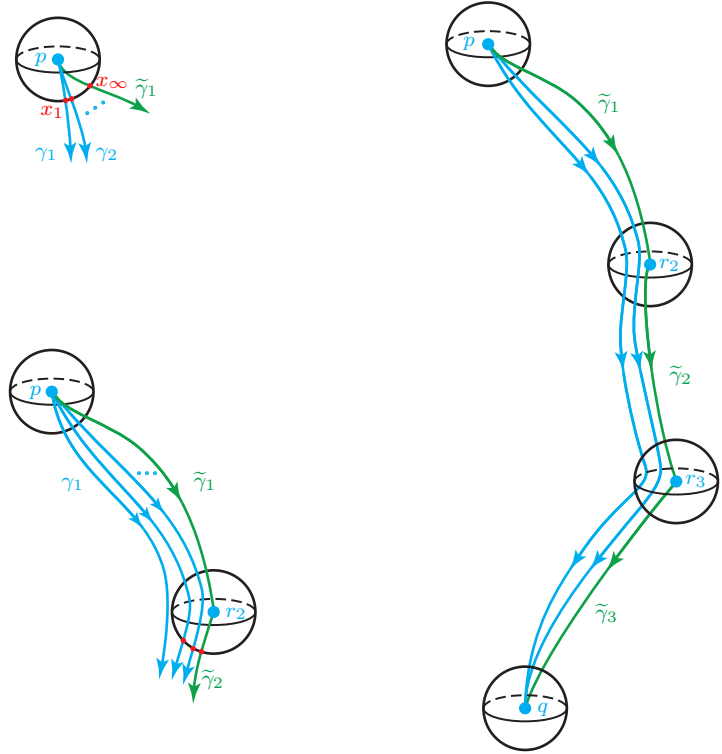 

So suppose $\lim_{t \to \infty}\widetilde\gamma_1(t)$ is the critical point $r_2\not= q$ (recall, we are taking $r_1=p$). Since $\mathcal{M}^p_{r_2}$ is non-empty, we know its dimension is greater than or equal to $0$ and so $\idx r_2< \idx p$. We can parameterize the $\gamma_i$ so that $\gamma_i(0)=x_i$ and by definition $\widetilde\gamma_1(0)=x_\infty$. We can now choose a small neighborhood $B_2$ of $r_2$ as above (that is, so that any flow line asymptotic to $r_2$ is transverse to $\partial B_2$, and other flow lines that intersect $B_2$ either intersect $\partial B_2$ transversely twice or are tangent to $\partial B_2$). 
We now know there is some $t_1>0$ such that $\widetilde{\gamma}_1(t_1)$ is the unique intersection point of $\widetilde{\gamma}_1$ with $\partial B_2$. Since the gradient flow is smooth and the initial conditions of the $\gamma_i$ converge to that of $\widetilde{\gamma}_1$, we know that there is some $N_1$ such that for $i\geq N_1$, $\gamma_i$ also intersects $\partial B_2$ transversely. Since the $\gamma_i$ do not limit to $r_2$ we know they must intersect $\partial B_2$ at another point. Denote this point by $x_i^2$. The sequence $\{x^2_i\}$ converges to a point $x^2_\infty$. We now let $\widetilde{\gamma}_2$ be the flow line through $x^2_\infty$. 

We can continue with this process until we find a $\widetilde{\gamma}_n$ that flows to $q$. 
\bhw
Prove this last statement. Specifically, why must one of the trajectories in the broken flow line eventually end at $q$?
\ehw
To see that the sequence $\gamma_i$ converges to the broken flow line $\cup_{i=1}^n  \widetilde{\gamma}_i$ we note that one may take the balls $B_i$ used to find the $\widetilde{\gamma}_i$ to be arbitrarily small, so points of the $\gamma_i$ in these balls are very close to the images of the $\widetilde{\gamma}_i$. There are finite intervals of time for which the $\gamma_i$ and $\widetilde{\gamma}_i$ are outside the $B_i$, and thus by the smoothness of the flow, we know for large enough $j$, the images of the $\gamma_j$ will be very close to the images of the $\widetilde{\gamma}_i$. 
\bhw
Carefully write out this convergence argument. 
\ehw
\begin{remark}
We note that the geometry of gradient flows that determines the ``compactness" result is simply the fact that flow lines of vector fields depend smoothly on their initial conditions. 
\end{remark}


\subsubsection{Morse Theorem~$4$} We now turn to the ``gluing" theorem. The basic idea here is to build an ``approximate flow line" and then prove there is a unique actual flow line associated with it. The second step is accomplished by a theorem of Floer \cite[Proposition~24]{Floer1995} who generalized prior results of Picard and Newton. More specifically, Newton gave an iterative scheme to find zeros of a real-valued function, and if one has an ``approximate zero" and one has appropriate control over the derivative of the function, then one can guarantee there is actually a zero near the approximate zero. Picard then gave an iterative scheme to find a solution to an initial value problem with similar control over where the solution is relative to an approximate solution. Floer finds a version of these ideas that can be applied to partial differential equations.  
We will not state Floer's result here as it is a bit technical, but we give an indication of why one would hope to be able to ``guess" an almost solution to an equation and then hope that there is an actual solution nearby. To this end, suppose we have a function $f\co \R\to \R$ and want to find a zero of $f$. Write $f$ as
\[
f(x)=f(0)+L(x) + N(x)
\]
where $L$ is a linear map (given by multiplication by $f'(0)$) and $N(x)$ is a non-linear map. Suppose we know that $f(x_0)$ is very small relative to $f'(0)$ (that is, the norm of $L$), then if $N$ were not there, we know there would be a zero of $f$ very close to $x_0$. Now, if $N(x)$ is very small in comparison to $L(x)$ for all $x$ near $x_0$, then one can also conclude that there is a zero of $f$ near $x_0$, and we can estimate how close the zero is to $x_0$ in terms of a bound on $L^{-1}N(x)$.  
\bhw
Work out the details of the above discussion. 
\ehw
The above gives a very simple indication of the type of information and bounds Floer uses to prove his result. 


We are now left to see how to construct an approximate flow line to a give broken flow line. Recall that we saw above that the space of parameterized flow lines $\underline{\mathcal{M}}^p_q$ is the preimage of the zero section of 
\[
\phi\co (\mathcal{P}^p_q\times P)\to E\co (\gamma, g)\mapsto \gamma'(t)+V(\gamma(t)).
\]
So we want to find an approximate zero of $\phi$ and then use a theorem of Floer to prove that there is an actual zero of $\phi$ nearby and actually parameterize the nearby zeros. 

Suppose $\widetilde{\gamma}_1\in\mathcal{M}^p_r$ and $\widetilde{\gamma}_2\in \mathcal{M}^r_q$ form a broken flow line. We can define a new path from $p$ to $q$ as follows. For a given positive number $\rho$ choose a parameterization of the flow lines so that $\widetilde{\gamma}_1(-1)$ and $\widetilde{\gamma}_2(1)$ are the points where the flow lines intersect the ball $B_2$ of radius $\rho$ about $r$ discussed in the previous section. Now the path $\widetilde{\gamma}_1\#_\rho\widetilde{\gamma}_2$ will be $\widetilde{\gamma}_1$ for $t\leq -1$, $\widetilde{\gamma}_2$ for $t\geq 1$ and for $t\in[-1,1]$ will be a (well-chosen) path in $B_2$. See Figure~\ref{fig:glue}.
\begin{figure}[htb]{
\begin{overpic}
{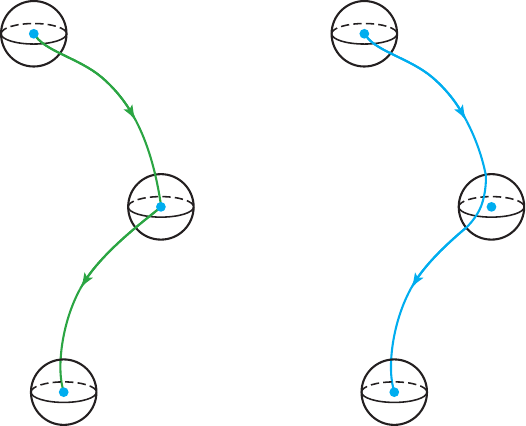}\footnotesize
\put(8, 188){\color{lblue}$p$}
\put(68, 104){\color{lblue}$r$}
\put(22, 14){\color{lblue}$q$}
\put(167, 188){\color{lblue}$p$}
\put(241, 104){\color{lblue}$r$}
\put(182, 14){\color{lblue}$q$}
\put(65, 160){\color{dgreen}$\gamma_1$}
\put(40, 55){\color{dgreen}$\gamma_2$}
\put(230, 150){\color{lblue}$\widetilde{\gamma}_1\#_\rho\widetilde{\gamma}_2$}
\end{overpic}}
\caption{On the left we see a broken flow line, and on the right we see an approximate flow line near the broken flow line.}
\label{fig:glue}
\end{figure}
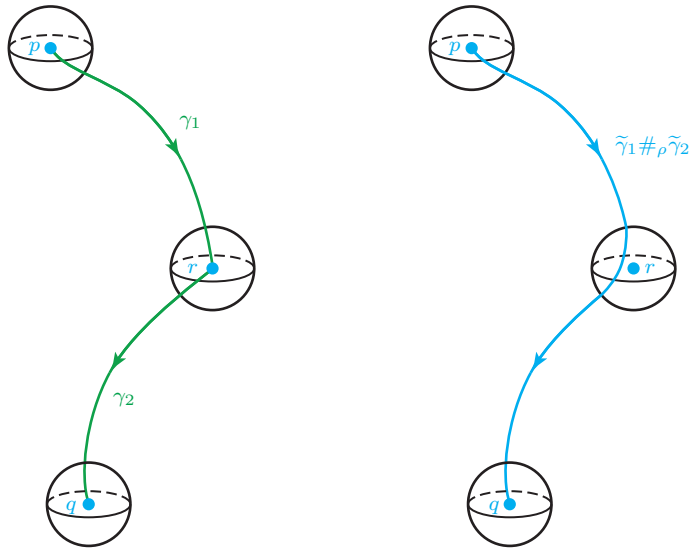

One can see that $\widetilde{\gamma}_1\#_\rho\widetilde{\gamma}_2$ is ``almost" a geodesic, in the sence that for most $t$ values it is a geodesic, and if one chooses the path correctly for $t\in[-1,1]$ then it is not far from being a geodesic on all of $\R$ (notice that this path can be chosen to have very small derivitive and the grandient is also very small here so their difference is very small). With some work, one may use an infinite-dimensional version of the above mentioned result to show that near $\widetilde{\gamma}_1\#_\rho\widetilde{\gamma}_2$ there is a $1$-parameter family of flow lines in $\mathcal{M}^p_q$ that converge to the broken flow line $\widetilde{\gamma}_1\cup\widetilde{\gamma}_2$.

\subsubsection{Another approach}\label{aa}
Given a critical point $p$ of a Morse function $f\co M\to \R$ and a metric $g$ on $M$, we can define the \dfn{unstable manifold} of $p$ to be 
\[
U_p=\{x\in M \text{ such that } \lim_{t\to -\infty}\gamma_x(t)=p\},
\]
and the \dfn{stable manifold} of $p$ to be
\[
S_p=\{x\in M \text{ such that } \lim_{t\to \infty}\gamma_x(t)=p\}.
\]
These are also sometimes referred to as the \dfn{descending manifold} and \dfn{ascending manifold}, respectively. One may prove, \cite{BanyagaHurtubise2004}, that $U_p$ and $S_p$ are both open disks of dimension $k$ and $n-k$, respectively, where $k=\idx p$.

We note that 
\[
\mathcal{M}^p_q=U_p\cap S_q.
\]
One must interpret this equation correctly. Specifically, $U_p\cap S_q$ is the image of the elements in $\mathcal{M}^p_q$. So we can define Morse homology by counting the flow lines in $U_p\cap S_q$. One can interpret Morse Theorem~$1$ as the statement that all the stable and unstable manifolds of the critical points of $f$ are transverse to each other. Some approaches to Morse homology do not define the moduli spaces $\mathcal{M}^p_q$, but instead focus on $U_p$ and $S_p$. While this is a fine approach, the one presented above gives a better introduction to other homology theories. The main reason we mention this approach is the following fact.
\bhw
Show that one may give $M$ the structure of a $CW$ complex using the unstable manifolds. 
\ehw
With this observation, one can easily identify Morse homology with cellular homology (which is, of course, known to be isomorphic to singular homology). 

\section{Lagrangian Floer homology}

One can think of Lagrangian Floer homology as an infinite-dimensional version of Morse homology applied to the ``action functional on the space of paths between two Lagrangian submanifolds", but it is not possible to carry out the details. The reason for this is that the ``gradient flow equation" for the action functional turns out to be an elliptic partial differential equation (as opposed to a parabolic equation which is the infinite dimensional analog of the finite dimensional gradient flow equation discussed above). Elliptic equations do not typically have solutions for ``initial value problems". More specifically, given a path between the two Lagrangian submanifolds, one cannot expect to find a solution to these equations, and thus the definitions of the moduli spaces above do not work in this case. Luckily, we can reformulate the above moduli spaces in terms of this elliptic equation. More specifically, Floer reformulated this idea in terms of pseudo-holomorphic curves. We will forgo this intuitive idea behind the theory and go straight to the definition of Lagrangian Floer theory that can be made rigorous. 

We begin by recalling a few basic facts from symplectic geometry to set the stage for Lagrangian Floer theory. The objects counted in the boundary map defining Lagrangian Floer theory are pseudo-holomorphic curves, so we discuss these in the next section. As we saw above, compactness issues are key to defining a homology theory, so in Section~\ref{Gc}, we discuss Gromov compactness, which details the behavior of limits of pseudo-holomorphic curves (under nice conditions). With this background in hand, we then turn to defining Lagrangian Floer homology in Section~\ref{LFh}, and then end with a section discussing applications of Lagrangian Floer homology. 

\subsection{A little symplectic geometry}
There are many good introductions to symplectic geometry, for example \cite{CannasDaSilva2001, McDuffSalamon98}. We refer the reader to those sources for a more comprehensive introduction to the subject, but discuss a few basic definitions and results here to set the stage for Lagrangian Floer homology. 

A \dfn{symplectic manifold} is a pair $(X,\omega)$, where $X$ is a manifold and $\omega$ is a $2$-form that is closed, that is $d\omega=0$, and non-degenerate, that is if at any point $x\in X$ have have that $\omega(v,w)=0$ for all $w\in T_xX$ implies that $v=0$. We call $\omega$ a symplectic form, or symplectic structure, on $X$. Some of the first things one observes about a symplectic manifold is that it dimension must be even, say $2n$ dimensional, and $\omega$ being non-degenerate is equivalently to $\omega^n$ being a volume form on $X$ (that is that $\omega^n$ is a non-zero $2n$-form at all points of $X$), where $\omega^n$ means the $n$-fold wedge product of $\omega$ with itself. 

\bex\label{stdex}\label{stdex}
The most basic example of a symplectic manifold is $(\R^{2n}, \omega_{std})$, where 
\[
\omega_{std}=\sum_{i=1}^n dx_i\wedge dy_i,
\] 
and $(x_1,y_1,\ldots, x_n, y_n)$ are Euclidean coordinates on $\R^{2n}$. One may easily check that $\omega_{std}=-d\lambda_{std}$ where 
\[
\lambda_{std}=\sum_{i=1}^n y_i \, dx_i,
\]
and so $\omega_{std}$ is clearly closed. It is also a simple exercise to see that $\omega^n$ is a volume form. So we see that $(\R^{2n}, \omega_{std})$ is indeed a sympletic manifold. 
\eex
In general, if a symplectic form $\omega$ is the differential of a $1$-form, then we call it an \dfn{exact symplectic form}. So the above example is an example of an exact symplectic form. 
\bex\label{canonicalcotangent}
Another key example of a symplectic manifold is $(T^*M,\omega_{can})$ where $T^*M$ is the cotangent bundle of any smooth manifold $M$ and $\omega_{can}$ is defined as follows. In local coordinates $(q_1,\ldots, q_n)$ on an open set $U$ in $M$ we can take local coordinates for $T^*U=U\times \R^n$ to be $(q_1,\ldots, q_n, p_1,\ldots, p_n)$ where the last $n$-coordinates are coordinates for $\R^n$ in the basis $\{dq_1,\ldots, dq_n)$. We can now define 
\[
\lambda_{can}=\sum_{i=1}^n p_i\, dq_i
\]
on $U$. 
\bhw
Show that this defines a $1$-form on all of $T^*M$. Specifically, cover $M$ in coordinate charts and show that the $\lambda_{can}$ defined above in each of these coordinate charts defines the same $1$-form whenever the coordinate charts overlap. 
\ehw
Thus we have $\lambda_{can}$ defined on all of $T^*M$ so we define $\omega_{can}=-d\lambda_{can}$. This gives a closed $2$-form that is easy to verify is also non-degenerate. 
\eex
Symplectic geometry originally grew out of an elegant reformulation of classical mechanics and has many beautiful and deep applications. For the sake of brevity, we refer the reader to the references mentioned at the beginning of this section for more details. 

A fundamental theorem in symplectic geometry, called the Darboux theorem, says that for any symplectic manifold $(X,\omega)$ and any point $x\in X$ there is a neighborhood $U$ of $x$ in $X$, a neighborhood $V$ of the origin in $\R^{2n}$, and a diffeomorphisms $\phi\co U\to V$ such that $\phi^*\omega_{std}=\omega$. That is, any symplectic manifold, locally, looks like the example above! We note that a diffeomorphism $\phi$ from one symplectic manifold, $(X,\omega)$, to another, $(X',\omega')$, is called a \dfn{symplectomorphism}.

A common means to build symplectomorphisms is via Hamiltonian flows. Specifically, given a function $H\co X\to \R$ on a symplectic manifold $(X, \omega)$ there is a unique vector field $v_H$ on $X$ such that
\[
\iota_{v_H}\omega= dH. 
\]
\bhw
Using the non-degeneracy of $\omega$, show that such a $v_h$ exists and is unique. 
\ehw
The function $H$ is called a \dfn{Hamiltonian} and the vector field $v_H$ is called a \dfn{Hamiltonian vector field}. Let $\phi_H\co M\to M$ be the time $1$ flow of $v_H$. One may easily check that $\phi_H$ is a symplectomorphism. We call $\phi_H$ a \dfn{Hamiltonian symplectomorphism}. We note that it is not important that we choose the time $1$ flow of $v_H$ because by scaling $H$ we scale $v_H$ and the time $1$ flow of these rescalled vector fields is simply the time $t$ flow of $v_H$ for some $t \not=1$. We note that on most symplectic manifolds, there are symplectomorphisms that are not Hamiltonian symplectomorphisms, but we do not pursue this matter here. 


The last ingredient we need for our discussion of Lagrangian Floer homology is Lagrangian submanifolds. A \dfn{Lagrangian submanifold} $L$ of a symplectic manifold $(X,\omega)$ is a submanifold of $X$ having half the dimension of $X$ and for which $\omega$ vanishes on the tangent space of $L$. 
\bex
Any oriented surface $\Sigma$ can be made a symplectic manifold simply by choosing an area form $\omega$ on $\Sigma$. We note that any $1$-dimensional manifold in $\Sigma$ will be Lagrangian. So we have \em{lots} of examples of Lagrangian submanifolds. 
\eex
\bhw
Show that the product of symplectic manifolds is symplectic in a natural way, and the product of Lagrangian submanifolds in each symplectic manifold is a Lagrangian submanifold of the product symplectic manifold. 
\ehw
From the previous exercise, we can find many Lagrangian tori in symplectic manifolds built as products of orientable surfaces. 

We end this section with an example showing the key nature of Lagrangian submanifolds in the study of symplectic manifolds. 
\bex
Let $(X,\omega)$ be a symplectic manifold and $\phi\co X\to X$ be a symplectomorphism. We now consider the symplectic manifold $(W,\omega_W)$ where $W=X\times X$ and $\omega_W=\omega\oplus -\omega$. To be a bit more precise about the definition of $\omega_W$ we let $\pi_i\co W\to X$ be the projection of $W$ to the $i^{th}$ factor of $W$ and then the more proper definition of $\omega_W$ is
\[
\omega_W=\pi_1^* \omega - \pi_2^* \omega.
\]
Now, if we let $\Gamma_\phi$ be the graph of $\phi$, that is $\Gamma_\phi=\{(x,\phi(x)) : x\in X\}$, then $\Gamma_\phi$ is a Lagrangian submanifold of $(W,\omega_W)$. Indeed, consider 
\[
\Phi\co X\to W\co x\mapsto (x,\phi(x)).
\] 
This map clearly parameterizes $\Gamma_\phi$ and the restriction of $\omega_W$ to the tangent space of $\Gamma_\phi$ is simply $\Phi^*\omega_W$. We now compute
\[
\Phi^*\omega_W = \Phi^*\pi_1^*\omega-\Phi^*\pi_2^*\omega=\omega-\phi^*\omega=0,
\]
since $\phi$ is a symplectomorphism. 
\eex
\bhw
Considering the previous example, show that symplectomorphisms of $(X,\omega)$ correspond to a certain subset of Lagrangian submanifolds of $(W,\omega_W)$. 
\ehw

\subsection{Pseudo-holomorphic curves}
Pseudo-holomorphic curves have been a cornerstone of the study of symplectic manifolds since their introduction by Gromov in \cite{Gromov85}, and, in particular, are used in the definition of Lagrangian Floer homology. In this section, we will discuss a few basic facts about pseudo-holomorphic curves. 

An \dfn{almost-complex structure} on a manifold $M$ is a bundle map
\[
J\co TM \to TM
\]
covering the identity on $M$ such that $J^2=-\text{id}_{TM}$. Notice that $J$ gives each tangent space $T_xM$ the structure of a complex vector space. 
\bex
Show that a complex manifold has a canonical almost-complex structure.
\eex
We note that not all almost-complex structures on a manifold come from a complex structure on the manifold. In particular, there are many manifolds that do not admit a complex structure, but do admit almost-complex structures. However, it is true that on an oriented surface $\Sigma$ complex structures and almost-complex structures are equivalent. 

Now, given a closed Riemannian surface $\Sigma$ then it has a natural complex structure (namely, one simply rotates a tangent vector anti-clockwise by $\pi/2$). We will denote this almost-complex structure $j$. Now given a manifold $M$ with an almost-complex structure $J$, we say a map 
\[
u\co \Sigma\to M
\]
is \dfn{pseudo-holomorphic} if it intertwines the almost-complex structures, that is 
\[
du\circ j = J\circ du
\]
where $du\co T\Sigma\to TM$ is the derivative of $u$. 
\bex
If $M=\C$ with the almost-complex structure coming from the complex structure on $\C$, then show that a function $u\co \Sigma\to \C$ is pseudo-holomorphic if and only if it is holomorphic in the sense of complex analysis. 
\eex

Below, we will need to have our complex structure related to a symplectic structure. To discuss this relation, we consider the symplectic manifold $(X,\omega)$. An almost-complex structure on $X$ is said to be \dfn{tamed} by $\omega$ if 
\[
\omega(v, Jv)>0,
\]
for any non-zero tangent vector $v$. We also say that $J$ is \dfn{compatible} with $\omega$ if $J$ is tamed by $\omega$ and $\omega$ is $J$-invariant, which means that $\omega(Jv,Jw)=\omega(v,w)$. The main result we will need about compatible almost-complex structure is the following. 
\begin{lemma}
The space of almost-complex structures compatible with a fixed symplectic structure is non-empty and convex. 
\end{lemma}
The proof of this lemma is straightforward, but a bit involved. It can be found in the references mentioned at the beginning of this section. 

\subsection{Gromov compactness}\label{Gc}
As we saw with Morse homology, understanding the compactness of the geometric objects counted in the homology boundary map is a key to making a homology theory. So, before we start to define Lagrangian Floer homology, we will discuss what can happen when considering sequences of pseudo-holomorphic curves. 

To this end, we first define ``singular" Riemannian surfaces. We start with a surface $\Sigma$ and a collection of properly embedded curves $\gamma_1,\ldots, \gamma_n$. See Figure~\ref{fig:precusp}. 
\begin{figure}[htb]{
\begin{overpic}
{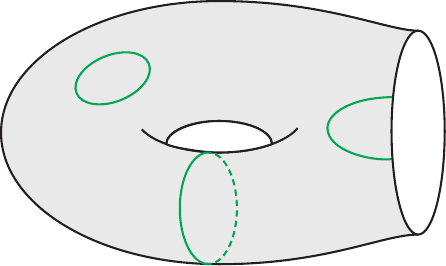}
\put(67, 76){\color{dgreen}$\gamma_1$}
\put(72, 25){\color{dgreen}$\gamma_2$}
\put(150, 52){\color{dgreen}$\gamma_3$}
\end{overpic}}
\caption{The properly embedded curves $\gamma_i$ on $\Sigma$.}
\label{fig:precusp}
\end{figure} 
Let $\widehat\Sigma$ be the result of collapsing each $\gamma_i$ to a point, which we denote by $z_i$. See Figure~\ref{fig:cusp}. We note that if $j$ is an almost-complex structure on $\Sigma$, then it induces one on $\widehat\Sigma-\{z_1, \cdots, z_n\}$, which we still denote by $j$. We call $(\widehat\Sigma, j)$ a \dfn{cusped curve}. 
\begin{figure}[htb]{
\begin{overpic}
{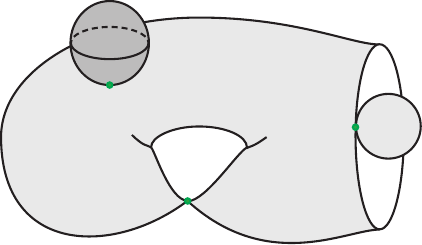}
\footnotesize
\put(50, 68){\color{dgreen}$z_1$}
\put(75, 20){\color{dgreen}$z_2$}
\put(160, 55){\color{dgreen}$z_3$}
\end{overpic}}
\caption{The cusped surface $\widehat\Sigma$ and the images $z_i$ of the $\gamma_i$.}
\label{fig:cusp}
\end{figure} 
We note there is a natrual quotient map $\pi\co \Sigma\to \widehat\Sigma$ and we define the \dfn{singular set} of $\widehat\Sigma$ to be $s(\widehat\Sigma)=\{ z_1,\ldots, z_n\}$. 

Now, given an oriented Riemannian surface $(\Sigma, j)$ we get a metric on $\Sigma$ by setting $h(v,w)=\Omega(v,jw)$ where $\Omega$ is any araa form on $\Sigma$ giving the orientation on $\Sigma$. Given an almost-complex structure $J$ compatible with a symplectic form $\omega$ on a manifold $X$ can now define the \dfn{action} of a map $u\co \Sigma\to X$ to be
\[
\mathcal{A}(u)=\int_\Sigma \| du\|^2 \Omega,
\]
where $\| du\|$ at the point $x$ is defined to be
\[
\| du\|_x = \text{max} \{\|du(v)\|_g: v\in T_x\Sigma, \|v\|_h=1 \}
\]
and $\|v\|_h$ is the length of $v$ measured with respect to the metric $h$ and similarly for $\|w\|_g$ where $g(v,w)=\omega(v, Jw)$. So pointwise $\|d u\|_x$ measures the {\em maximal distortion of vectors in $T_x\Sigma$ by $du_x$} and $\mathcal{A}(u)$ is just the integral of this over $\Sigma$. We note that one can similarly define the action for a map from a cusped curve to $X$. 

\bhw
Show that 
\[
\mathcal{A}(u)\leq \int_\Sigma u^*\omega,
\]
and that one has equality if and only if $u$ is pseudo-holomorphic. 
\ehw

Given a sequence of pseudo-holomorphic curves $u_k\co (\Sigma, j_k) \to (X,J)$ and a cusped curve $\pi:\Sigma \to \widehat\Sigma$ with a almost-complex structure $j$, then we say the sequence $\{u_k\}$ converges to $u_\infty\co \widehat\Sigma\to X$ if 
\begin{enumerate}
\item there are maps $\phi_k\co \Sigma\to \widehat\Sigma$ and curves $\gamma_1^{k}, \ldots, \gamma_n^{k}$ such that
\[
\phi_k|_{\Sigma-\cup \gamma_i}\co (\Sigma-\cup \gamma_i)\to (\widehat \Sigma-s(\widehat\Sigma))
\] 
is a diffeomorphism,
\item the sequence of almost-complex structures $\{(\phi_k|_{\Sigma-\cup \gamma_i}^* j_k\}$ $C^\infty$-converges to $j$ on compact subsets as $k\to \infty$, 
\item the sequence $\{u_k\circ\phi_k^{-1}|_{\widehat\Sigma-s(\widehat\Sigma)}\}$ $C^\infty$-converges to $u_\infty|_{\widehat\Sigma-s(\widehat\Sigma)}$ on compact subsets as $k\to \infty$, and 
\item $\lim_{k\to \infty} \mathcal{A}(u_k)=\mathcal{A}(u_\infty)$.  
\end{enumerate}
\begin{theorem}[Gromov compactness, \cite{Gromov85}]\label{gct}
Let $(X, \omega)$ be a symplectic manifold with an almost-complex structure $J$. If $u_k\co (\Sigma, j_k)\to (X,J)$ is a sequence of pseudo-holomorphic maps with bounded action (here if, $\partial \Sigma\not=\emptyset$ then $u_k(\partial \Sigma)\subset L$ for some fixed Lagrangian submanifold $L$ of $X$), then some subsequence of $\{u_k\}$ converges to a cusp curve $u_\infty:(\widehat\Sigma, j)\to (X,J)$. (We note that $s(\widehat\Sigma)$ could be empty, in which case the $u_k$ converge to a pseudo-holomorphic curve.)
\end{theorem}
\begin{remark}
We note that Gromov's original theorem did not require a symplectic form, but merely a Riemannian metric, so that the action of the pseudo-holomorphic curves can be measured. We state this version as it is all that we will need below.
\end{remark}

To get some sense of this type of convergence, we look at an example. 
\bex
Let $\Sigma=S^2$ be the $1$-point compactification of $\C$ (with any fixed area form and almost-complex structure) and consider the maps
\[
u_k\co S^2\to S^2\times S^2\co z\mapsto (z,1/kz).
\]
We can see the graphs of these functions in Figure~\ref{fig:s2s2}.
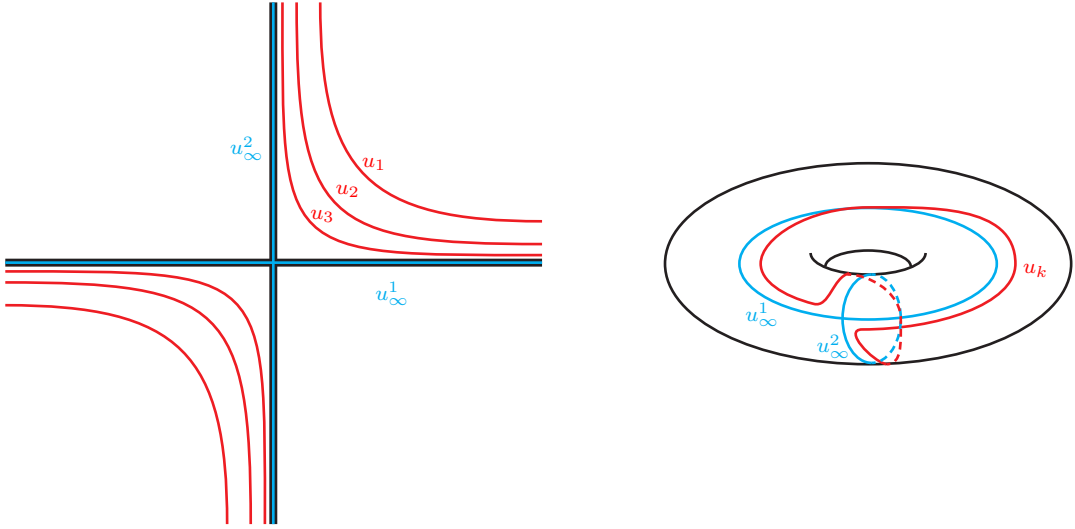
\begin{figure}[htb]{
\begin{overpic}
{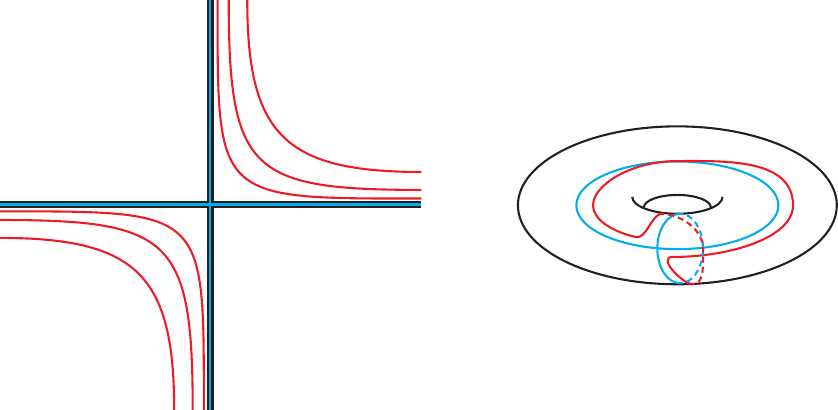}
\footnotesize
\put(135, 135){\color{red}$u_1$}
\put(125, 125){\color{red}$u_2$}
\put(115.5, 115.5){\color{red}$u_3$}
\put(140, 85){\color{lblue}$u_\infty^1$}
\put(85, 140){\color{lblue}$u_\infty^2$}
\put(385, 95){\color{red}$u_k$}
\put(280, 77){\color{lblue}$u_\infty^1$}
\put(307, 65){\color{lblue}$u_\infty^2$}
\end{overpic}}
\caption{On the left are the images of the $u_k$, shown in red, and the $u_\infty^l$, shown in blue. On the right we see the same picture where the spheres have not been punctured (though they are still, necessarily, half-dimensional, so the figure can be drawn).}
\label{fig:s2s2}
\end{figure} 
We notice that as $k\to \infty$ these maps converge to $u_\infty^1\co S^2\to S^2\times S^2:z\mapsto (z,0)$ way from $z=0$. The problem is that the gradient of the $u_k$ is blowing up at $z=0$. So to see what the $u_k$ converge to in a neighborhood of $0$ we need to reparameterize the $u_k$. This is similar to when we studied geodesics, and we were only interested in the image of the curves and not their parameterization. If we reparameterize the $u_k$ by the holomorphic map $\phi\co S^2\to S^2: z\mapsto 1/kz$ we see
\[
u_k\circ\phi\co S^2\to S^2\times S^2\co z\mapsto (1/kz, z).
\]
And these maps converge to $u_\infty^2\co S^2\to S^2\times S^2\co z\mapsto (0,z)$. So the sequence of pseudo-holomorphic curves $u_k(S^2)$ converges to the cusped curve $u_\infty:\widehat{S^2}\to S^2\times S^2$ where $\widehat{S^2}$ is a wedge of two $S^2$s and $u_\infty$ is $u_\infty^1$ on one of the spheres and $u_\infty^2$ on the other. 

We note in this simple example that we were able to find one global reparameterization to ``see'' both the limiting spheres, but in general, one might need to reparameterize each of the $u_k$ separately, and this is the role of the $\phi_k$ in the definition of convergence. 
\eex
\subsection{Defining Lagrangian Floer homology}\label{LFh}
We are now ready to define Lagrangian Floer homology. We start with a symplectic manifold $(X,\omega)$ and two Lagrangian submanifolds $L_0$ and $L_1$ of $X$. Lagrangian Floer homology will be an invariant of the pair $L_0$ and $L_1$ up to Hamiltonian isotopy of either of the Lagrangian submanifolds. 

We first note that one may slightly isotop the $L_i$ so that they are transverse to each other. The chain complex for Lagrangian Floer homology is 
\[
FC(L_0,L_1)=\text{ the free $R$-module generated by } L_0\cap L_1,
\]
where $R$ is some ring of field like $\Z$ or $\Z/2\Z$.

To define the boundary map, we will consider ``pseudo-holomorphic strips". That is we will consider the surface $\R\times [0,1]$ with coordinate $(s,t)$ and the complex structure that sends $\frac{\partial}{\partial t}$ to $\frac{\partial}{\partial s}$ and $\frac{\partial}{\partial s}$ to $-\frac{\partial}{\partial t}$. Thus, if we choose an almost-complex structure $J$ on $X$ that is compatible with $\omega$ (this is like choosing a Riemannian metric when studying Morse homology), then a map $u\co(\R\times [0,1])\to X$ is pseudo-holomorphic if and only if 
\[
\frac{\partial u}{\partial s} + J \frac{\partial u}{\partial t}=0
\]
where of course $\frac{\partial u}{\partial s}$ is the image of the vector $\frac{\partial}{\partial s}$ under $du$ and similarly for $\frac{\partial u}{\partial t}$.
\bhw
Prove the statement characterizing when $u$ is pseudo-holomorphic. 
\ehw
We can also consider a family of almost complex structures $J_t$ compatible with $\omega$ and consider pseudo-holomorphic strips $u\co (\R\times[0,1])\to X$ that satisfy 
\[
\frac{\partial u}{\partial s} + J_t \frac{\partial u}{\partial t}=0.
\]
This extra flexibility in the definition of a pseudo-holomorphic strip can be helpful in proving various technical results.

Now given two point $x,y\in L_0\cap L_1$ we consider the space
\begin{align*}
\mathcal{M}^x_y=\{ u\co (\R\times [0,1])&\to X \text{ such that}\\
&\text{1) $u(s,i)\in L_i $ for $i=0,1$},\\
&\text{2) $\lim_{s\to -\infty}u(s,t)=x$},\\
&\text{3) $\lim_{s\to \infty}u(s,t)=y$}, \text{ and}\\
&\text{4)  $u$ is pseudo-holomorphic}
\}/\R.
\end{align*}
To understand the $\R$ action above we note that $T_{s_0}\co (\R\times [0,1])\to (\R\times [0,1])\co (s,t)\mapsto (s+s_0,t)$ is a holomorphic map so given any pseudo-holomorphic $u\co (\R\times[0,1])\to X$ the map $u\circ T_{s_0}$ is also pseudo-holomorphic and hence $\R$ acts on the set of such psedudo-holomorphic maps. 

It will be important for us to keep track of the homotopy class of an element $u\in \mathcal{M}^x_y$. To this end we define the set
\begin{align*}
\pi_2(x,y)=\{ \text{homotopy classes of maps }&  u\co (\R\times [0,1])\to X \text{ such that}\\
&\text{1) $u(s,i)\in L_i $ for $i=0,1$},\\
&\text{2) $\lim_{s\to -\infty}u(s,t)=x$}, \text{and}\\
&\text{3) $\lim_{s\to \infty}u(s,t)=y$}
\}
\end{align*}
For each $h\in \pi_2(x,y)$ we can now consider the spaces
\[
\mathcal{M}^x_y(h)=\{u\in \mathcal{M}^x_y \text{ such that } u\in h\}
\]
and notice that
\[
\mathcal{M}^x_y=\bigcup_{h\in \pi_2(x,y)} \mathcal{M}^x_y(h).
\]

Before we can state the main theorems to construct our boundary map, we need to consider an ``index" for $x$ and $y$. In Morse theory, we could assign an index to each critical point, and we could then determine the dimension of the moduli space of flowlines in terms of these indices. Unfortunately, we do not have a well-defined index for each $x\in L_0\cap L_1$, but we do have a ``relative index" for a pair of points $x,y\in L_0\cap L_1$, and this will be sufficient to define the boundary map for our chain complex. This relative index is provided by the Maslov index or Maslov class. 
\subsubsection{The Maslov class}\label{maslovsection}
We begin by considering $\R^{2n}=\C^n$ with its standard symplectic structure $\omega_{std}$ defined in Example~\ref{stdex}. We can now let $\mathcal{L}(\C^n)$ be the space of linear Lagrangian subspaces in $\C^n$. It is not too hard to show
\[
\pi_1(\mathcal{L}(\C^n))\cong \Z
\]
and there is a ``preferred" generator, see \cite{RobbinSalamon1993}. 
\bhw
Prove this for when $n=1$. \\ Hint: Notice that $\mathcal{L}(\C)$ is simply $\R P^2$ and a natural generator for its fundamental group is the loop that starts with the $x$-axis and rotates through an angle of $\pi$. 
\ehw
Now consider two linear Lagrangian subspaces $\Lambda_0$ and $\Lambda_1$ in $\C^n$ that intersect transversely. 
\bhw
Show that there is a complex linear map $I\co \C^n\to \C^n$ such that $I(\Lambda_1)=\Lambda_0$.
\ehw
Using $I$ from the exercise, we can define a map a path in $\mathcal{L}(\C^n)$ from $\Lambda_1$ to $\Lambda_0$ given by
\[
\gamma_{\Lambda_0,\Lambda_1}(t)= e^{tI} \Lambda_1= \left( (\cos t)\text{ id}_{\C^n} + (\sin t) I\right) \Lambda_1,
\]
as $t$ goes from $0$ to $\pi/2$.
See Figure~\ref{fig:crotate}.
\begin{figure}[htb]{
\begin{overpic}
{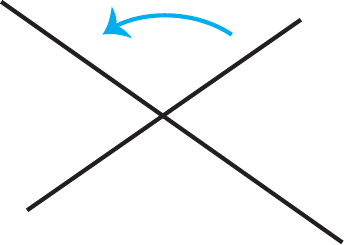}
\put(15, 85){$\Lambda_0$}
\put(135, 85){$\Lambda_1$}
\end{overpic}}
\caption{Two Lagrangian subspaces of $\C$. The map $I$ is indicated by the arrow. The path $\gamma_{\Lambda_0,\Lambda_1}$ rotates $\Lambda_1$ anticlockwise to $\Lambda_0$.}
\label{fig:crotate}
\end{figure}

Now suppose that $u\co \R\times [0,1]\to X$ is a representative of an element in $\pi_2(x,y)$. Consider 
\[
u^*TX \cong (\R\times[0,1])\times \C^n.
\]
Thus, if we have a linear Lagrangian subspace in $T_xX$ for some $x$ in the image of $u$, then we can use the trivialization above to think of it as a Lagrangian subspace of $\C^n$, that is, as an element of $\mathcal{L}(\C^n)$. Thus restricting $u$ to $\R\times\{i\}$, for $i=0,1$, will give paths $\gamma_i(t)=T_{u(t,i)}L_i$ in $\mathcal{L}(\C^n)$. We can think of extending the paths to $[-\infty, \infty]$ by continuity. So $\gamma_i(\infty)$ will be $T_yL_i$ and $\gamma_i(-\infty)$ will be $T_xL_i$. We can finally define the path
\[
\gamma_u=\gamma_0*\gamma_{\gamma_1(\infty),\gamma_0(\infty)}* \overline{\gamma_1}*\overline{\gamma_{\gamma_0(-\infty),\gamma_1(-\infty)}}
\]
by concatenating the paths, where $*$ denotes concatenation and $\overline{\gamma}$ is the path $\gamma$ run backwards. See Figure~\ref{fig:maslovex}. 
\begin{figure}[htb]{
\begin{overpic}
{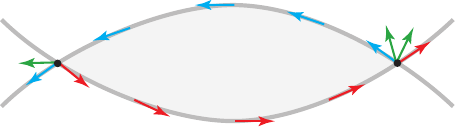}
\put(225, 48){$\Lambda_0$}
\put(225, 8){$\Lambda_1$}
\put(25, 37){$x$}
\put(188, 17){$y$}
\end{overpic}}
\caption{The gray is the image of a pseudo-holomorphic curve in $\C$ with boundary on $\Lambda_0$ and $\Lambda_1$. We drew red and blue arrows whose span gives the tangent space to the $\Lambda_i$. The green arrows give the interpolation between the end of the red and blue arrows given by the paths $\gamma_{\gamma_1(\infty),\gamma_0(\infty)}$ and $\overline{\gamma_{\gamma_0(-\infty),\gamma_1(-\infty)}}$. (Note that the first path rotates anticlockwise a bit more than $\pi/2$ while the second path rotates clockwise a bit less than $\pi/2$.) This gives a loop in $\mathcal{L}(\C)$, which we know to be $\R P^2$ from the exercise above. The Masolv index of this disk is $1$. }
\label{fig:maslovex}
\end{figure}
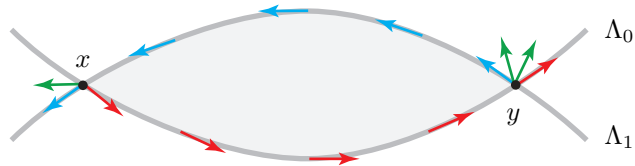 
It is clear that $\gamma_u$ is a loop in $\mathcal{L}(\C^n)$ and so we define the \dfn{Maslov index} of $u$ to be the integer associated to the homotopy class of this loop
\[
\mu(u)\in \Z\cong \pi_1(\mathcal{L}(\C^n))
\]

We are now ready to state the first main theorem.
\begin{tcolorbox}[title={Floer Theorem 1 (Transversality)}]
Suppose $L_0$ and $L_1$ are Lagrangian submanifolds of $(X,\omega)$ that intersect transversely. 
For a generic choice of almost complex structure $J$ on compatible with $\omega$ the space 
\[
\mathcal{M}^x_y(h)
\]
is a manifold of dimension $\mu(u)$ for some $u$ the class $h\in \pi_2(x,y)$, were $x$ and $y$ are intersection points of $L_0$ and $L_1$.
\end{tcolorbox}

One can also orient these moduli spaces. We note in the satement we need the Lagrangian submanifolds to be spin. {Spin is a mild condition on a manifold that is not relevant for our main discussion, so we do not discuss it further.} 
\begin{tcolorbox}[title={Floer Theorem 2 (Orientability)}]
Suppose $L_0$ and $L_1$ are spin Lagrangian submanifolds of $(X,\omega)$ that intersect transversely. 
There is a natural way to orient the moduli spaces
\[
\mathcal{M}^x_y(h).
\]
\end{tcolorbox}

\subsubsection{Compactness and gluing}

We have discussed compactness for pseudo-holomorphic curves in almost complex manifolds in Section~\ref{Gc}, but we now consider the specific pseudo-holomorphic curves we will need in the definition of Floer homology. 
\begin{tcolorbox}[title={Floer Theorem 3 (Compactness)}]
A sequence of pseudo-holomorphic curves $(u_n,j_n)$ in $\mathcal{M}^x_y(h)$ that has $\mathcal{A}(u_n)$ bounded will have a subsequence that converges to either a pseudo-holomorphic curve in $\mathcal{M}^x_y(h)$ or a cusped pseudo-holomorphic curve as indicated in Figure~\ref{fig:stripcusp}. 
\end{tcolorbox}
\begin{figure}[htb]{
\begin{overpic}
{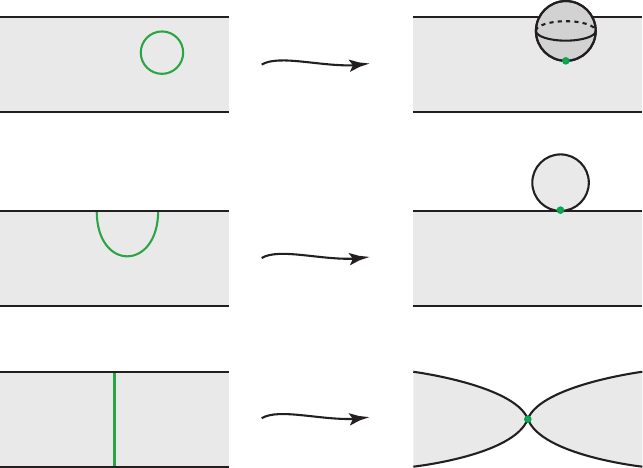}
\footnotesize
\put(-20, 192){(1)}
\put(-20, 98){(2)}
\put(-20, 22){(3)}
\end{overpic}}
\caption{Cusped curves to which a pseudo-holomorphic strip can converge. In the first example, a sphere ``bubbles off'' in the limit. In the second example, a disk ``bubbles off'' in the limit. In the last case, the pseudo-holomorphic strip limits to a ``broken strip''. }
\label{fig:stripcusp}
\end{figure} 

Finally, we have a gluing theorem. We state a vague version of the theorem as the full theorem can be a bit technical. 
\begin{tcolorbox}[title={Floer Theorem 4 (Gluing)}]
Any cusped pseudo-holomorphic strip (that is, anything that can appear in the limit in Floer Theorem 3) can be ``glued'' to find a family of pseudo-holomorphic strips that converge to it. 
\end{tcolorbox}

We are now ready to return to the definition of Floer homology. Recall from the beginning of this section that given two Lagrangian submanifolds $L_0$ and $L_1$ of a symplectic manifold $(X,\omega)$ (after isotoping them so they are transverse to each other), we defined the chain complex for Lagrangian Floer homology to be
\[
FC(L_0,L_1)=\text{ the free $R$ module generated by } L_0\cap L_1.
\]
We now define the chain map
\[
\partial\co FC(L_0,L_1) \to FC(L_0,L_1)
\]
on an element $p_+\in L_0\cap L_1$ by
\[
\partial p_+=\sum_{p_-\in L_0\cap L_1}\left( \sum_{h\in \pi_2(p_+,p_-), \mu(h)=1} |\mathcal{M}^{p_+}_{p_-}(h)/\R| \, p_-\right).
\]
Recall that $\mathcal{M}^{p_+}_{p_-}(h)$ has an $\R$ actions so $\mathcal{M}^{p_+}_{p_-}(h)/\R$ is just its quotient by this action and $|\mathcal{M}^{p_+}_{p_-}(h)/\R|$ is the number of points in the $0$-dimesnional manifold $\mathcal{M}^{p_+}_{p_-}(h)/\R$.
\begin{remark}
For this sum to be finite, we need 
\begin{equation}\tag{$*$}
\{\mathcal{A}(\phi) : [\phi]\in \pi_2(p_+,p_1)\}<\infty
\end{equation}
This will happen if $|\pi_2(p_+,p_-)|$ is finite. 
\end{remark}
We now see when $\partial$ is a differential.
\begin{lemma}
The map $\partial$ above is well-defined if $(*)$ is holds. Moreover, if the limits (1) and (2) in Figure~\ref{fig:stripcusp} do not occur, then $\partial^2=0$.
\end{lemma}
\begin{proof}
To see that the map is well-defined, we need to see that $|\mathcal{M}^{p_+}_{p_-}(h)/\R|$ is compact (it is already a $0$-manifold, so it will then be finite). Since we are assuming $(*)$ we can use the Gromov Compactness Theorem, Theorem~\ref{gct}, to see that a sequence of pseudo-holomorphic strips in $\mathcal{M}^{p_+}_{p_-}(h)$ that does not converge to a pseudo-holomorphic strip must converge to a cusped strip as in Figure~\ref{fig:stripcusp}. But we are assuming that the limits (1) and (2) do not occur, so the only limit we can have is (3). That is, our sequence $u_n$ converges to a broken strip. That is it converges to the union of two strips, say $u_\infty^1$ and $u_\infty^2$. 
\bhw
Show that
\[
\mu(u_n)=\mu(u_\infty^1)+\mu(u_\infty^2).
\]
\ehw
Since we know that $\mu(u_n)=1$, we see that one of the $u_\infty^i$ must be in a $0$-dimensional moduli space. But we know any pseudo-holomorphic strip is in a moduli space with an $\R$-action, so the space cannot have dimension $0$. Thus, our sequence cannot limit to (3) and must limit to another pseudo-holomorphic strip. In particular, $\mathcal{M}^{p_+}_{p_-}(h)$ is compact.


The proof that $\partial^2=0$ is analogous to the proof given in Theorem~\ref{d2} for $\partial^2=0$ in Morse homology. Specifically, see Figure~\ref{fig:picproof2}.
\begin{figure}[htb]{
\begin{overpic}
{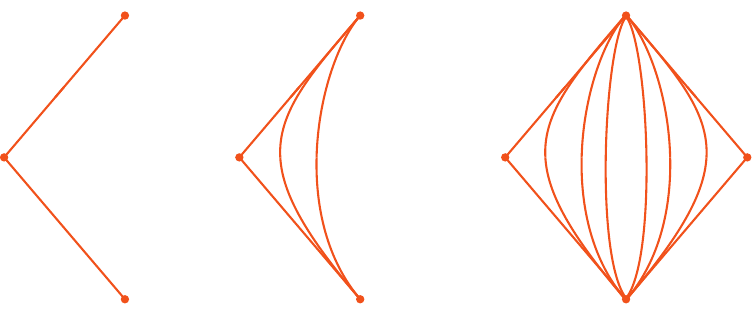}
\put(58, 151){$p_+$}
\put(-7, 75){$r$}
\put(58, -3){$p_-$}
\put(171, 151){$p_+$}
\put(105, 75){$x$}
\put(171, -3){$p_-$}
\put(298, 151){$p_+$}
\put(234, 75){$x$}
\put(298, -3){$p_-$}
\put(364, 75){$y$}
\end{overpic}}
\caption{The proof that $\partial^2=0$.}
\label{fig:picproof2}
\end{figure}
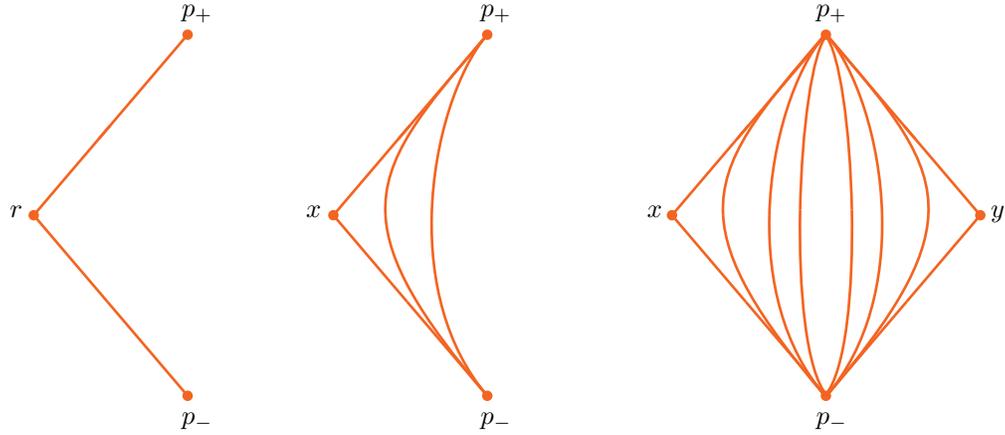 
There we see that if $p_-$ is in $\partial^2 p_+$ (that is $p_-$ is in $\partial x$ where $x$ is in $\partial p_+$), then the gluing theorem says there are nearby pseudo-holomorphic strips from $p_+$ to $p_-$ that converge to the strip from $p_+$ to $x$ and from $x$ to $p_-$. Now, since limits (1) and (2) are not allowed, the only compactification of this moduli space is by another broken strip through some double point $y$. 
\end{proof}
\begin{remark}
Notice that if the limit (2) in Figure~\ref{fig:stripcusp} could occur, then the last diagram in Figure~\ref{fig:picproof2} could look like Figure~\ref{fig:d2not0}. 
\begin{figure}[htb]{
\begin{overpic}
{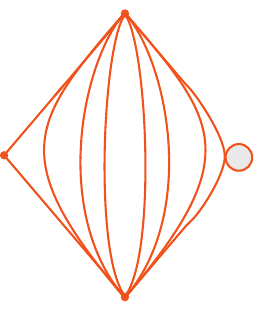}
\put(58, 151){$p_+$}
\put(-8, 75){$x$}
\put(58, -3){$p_-$}
\end{overpic}}
\caption{How $\partial^2=0$ can go wrong if one has undesirable limits. The limit on the right is supposed to represent a disk bubling off.}
\label{fig:d2not0}
\end{figure} 
In that case, $p_-$ might only show up in $\partial^2 p_+$ an odd number of times, and hence $\partial^2 p_+\not=0$. 
\end{remark}

Given Lagrangian submanifolds $L_0$ and $L_1$ satisfying $(*)$ and for which we do not have limits (1) and (2) in Figure~\ref{fig:stripcusp}, we can define the \dfn{Floer homology} of $L_0$ and $L_1$ to be 
\[
FH(L_0,L_1)=\frac{\ker (\partial \co FC(L_0,L_1)\to FC(L_0,L_1)))}{\img (\partial \co FC(L_0,L_1)\to FC(L_0,L_1))}
\]
\begin{remark}
We note that one can define a relative grading on $FH(L_0, L_1)$ using the Maslov class, but this is a little involved (it might be only well-defined modulo some integer), so we do not address this here. 
\end{remark}

Recall, earlier we discussed that given a function $H\co X\to \R$, there is a Hamiltonian vector field $v_H$ defined by $\iota_{v_H}\omega=dH$. Any time $t$ flow of $v_H$ is called a \dfn{Hamiltonian diffeomorphism}. 
\begin{theorem}
Given a symplectic manifold $(X,\omega)$ and Lagrangian submanifolds $L_0$ and $L_1$ such that the limits (1) and (2) in Floer Theorem~3 do not occur. If $(*)$ holds, then $FH(L_0,L_1)$ is independent of the almost complex structures $J_t$ chosen and 
\[
FH(\phi(L_0), \psi(L_1))=FH(L_0, L_1)
\]
for any Hamiltonian diffeomorphisms $\phi$ and $\psi$. 
\end{theorem}
The proof of this theorem is similar to the proof that Morse homology is independent of the chosen metric and Morse function. 

\subsection{Applications}
In this section, we will consider a simple situation when Floer homology is well-defined and a surprising construction that allows one to define an invariant of $3$-manifolds using Lagrangian Floer homology. 

\subsubsection{Exact Lagrangians submanifolds}
Let $(X,\omega)$ be an exact symplectic manifold, meaning that there is a $1$-form $\lambda$ such that $d\lambda=\omega$. 

Suppose that $\Sigma$ is a closed pseudo-holomorphic surface in $X$. Then $\omega(T_x\Sigma)>0$ for all $x\in \Sigma$ (since, for any $v\in T_x\Sigma$ we know $v$ and $Jv$ span $T_x\Sigma$ and $\omega(v, Jv)>0$ since $J$ is tamed by $\omega$). Thus, we see that 
\[
\mathcal{A}(\Sigma)=\int_\Sigma \omega>0,
\] 
but
\[
\int_\Sigma \omega=\int_\Sigma d\lambda=\int_{\partial \Sigma} \lambda=0.
\]
This contradiction indicates that there are no closed pseudo-holomorphic surfaces in $X$ and thus limit (1) in Floer Theorem~3 cannot happen. 

Now suppose that $L$ is an \dfn{exact Lagrangian submanifold} of $X$. This means that there is some function $f\co L\to \R$ such that $\lambda|_L=df$. Given a holomorphic sufacse $\Sigma$ with boundary on $L$, we note as above that 
\[
\int_\Sigma \omega>0,
\]
but 
\[
\int_\Sigma \omega=\int_{\partial \Sigma} \lambda=\int_{\partial \Sigma} df= \int_{\partial (\partial \Sigma)} f=0.
\]
This contradiction shows that limit (2) in Floer Theorem~3 cannot happen. 

Finally, if $L_0$ and $L_1$ are two exact Lagrangian submanifolds of $X$, with corresponding functions $f_0$ and $f_1$, and $u\co (\R\times [0,1])\to X$ is a pseudo-holomorphic strip with boundary on $L_0$ and $L_1$ from the intersection points $p_+$ to $p_-$, then 
\begin{align*}
\mathcal{A}(u)&=\int_{\R\times [0,1]}u^*\omega=\int_{\partial(\R\times [0,1])} u^*\lambda=\int_{-\infty}^\infty u^*f_1 -\int_{-\infty}^\infty u^*f_0\\
&=f_1(p_+)-f_1(p_-)-f_0(p_+)+f_0(p_-)
\end{align*}
Thus, $\mathcal{A}$ is bounded on $\pi_2(p_+,p_-)$ and so $(*)$ is true. 

From the above, we see the following result holds.
\begin{theorem}
Floer homology is well-defined for exact Lagrangian submanifolds in an exact symplectic manifold. 
\end{theorem}

Now consider a Morse function $f\co M\to \R$ on a manifold $M$. Let $\pi\co T^*M\to M$ be the projection map from the cotangent bundle of $M$ to $M$. If we set $F=f\circ \pi\co T^*M\to M$ we can let $v_F$ be the Hamiltonian vector field associated to $v_F$ on $T^*M$ using the canonical symplectic structure on $T^*M$ from Examples~\ref{canonicalcotangent}. 

 Now let $L$ be the image of the zero section $Z$ under the time $1$ flow of $v_F$. It is easy to see that $Z$ and $L$ are exact Lagrangian submanifolds of the exact symplectic manifold $T^*M$. Thus, Floer homology is well defined. Floer \cite{Floer1989} showed that rigid flow lines of the gradient of $f$ are in correspondence with pseudo-holomorphic strips in the Floer complex of $L$ and $Z$. From this, one can show that 
 the Floer homology of $Z$ and $L$ is isomorphic to the Morse homology of $M$:
  \[
 FH(Z,L)\cong H(M;\Z).
 \]
 This leads to the following fundamental result.
 \begin{theorem}
 If $\phi$ is a Hamiltonian diffeomorphism of $T^*M$ then 
 \[
 |(\phi(Z)\cap Z)|\geq \sum_{i=0}^n \dim(H_i(M;\Z))
 \]
 where $n$ is the dimension of $M$. 
 \end{theorem}
 \begin{proof}
 We know that the dimension of a chain complex is greater than or equal to the dimension of its homology, and the dimension of $FC(Z,L)$ is $ |(\phi(Z)\cap Z)|$, so
 \[
  |(\phi(Z)\cap Z)|\geq \dim FH(Z,L)\cong \dim H(M;\Z).
 \]
 This establishes the theorem.
 \end{proof}
 
 We now have the following famous corollary of Gromov \cite{Gromov85} (though Gromov's proof was very different). 
\begin{theorem}
There is no closed exact Lagrangian submanifold in $(\R^{2n}, \omega_{std})$. (See Examples~\ref{stdex} for the symplectic structure on $\R^{2n}$.)
\end{theorem}
\begin{proof}
Suppose there was a closed exact Lagrangian submanifold $L$ in $(\R^{2n}, \omega_{std})$. Let $f\co \R^{2n}\to \R$ be the projection onto the $x_1$-axis. Then the flow of the Hamiltonian vector field $v_f$ is translation in the $y_1$-direction. Thus there is a Hamiltonian diffeomorphism $\phi\co \R^{2n}\to \R^{2n}$ such that $\phi(L)\cap L=\emptyset$. This implies that $FH(L,\phi(L))=0$ but from the previous theorem, and the fact that any Lagrangian submanifold has a neighborhood symplectomorphic to a neighborhood of the zero section in its cotangent bundle, we know that $FH(L, \phi(L))$ must be isomorphic to $H(L;\Z)$ which is not trivial. Thus, no such $L$ exists. 
\end{proof}
 
 \subsubsection{Heegaard Floer homology}
 We now discuss an application of Lagrangian Floer homology to $3$-manifold topology. We only give a hint at this application and refer the reader to \cite{OzsvathSzab02004} for all the details.  
 
 Given a closed oriented $3$-manifold $M$, there is always a surface $\Sigma$ of some genus $g$ that splits $M$ into two genus $g$ handlebodies. By this we mean that $M\setminus \Sigma=V_0\cup V_1$ and there are $g$ disjoint curves $\alpha_1,\ldots, \alpha_g$ on $\Sigma$ such that each $\alpha_i$ bounds a disk $A_i$ in $V_0$ such that $V_0\setminus \cup_{i=1}^g A_i$ is a ball, and simlarly there are disjoint curves $\beta_1,\ldots, \beta_g$ on $\Sigma$ such that each $\beta_i$ bounds a disk $B_i$ in $V_1$ such that $V_1\setminus\cup_{i=1}^g B_i$ is a ball. This is called a \dfn{Heegaard splitting of $M$}. We note that the triple $(\Sigma,\{\alpha_i\}, \{\beta_i\})$ determines $M$, so if we can associate an invariant to this tripple then it might determine an invariant of $M$. It turns out that different triples can give the same $M$, so to really get an invariant of $M$ we need to make sure that all triples for $M$ give the same result. 
 
Notice that choosing an area form on $\Sigma$, we see that $\Sigma$ is a symplectic manifold and that any curve on $\Sigma$ is Lagrangian. Thus the $g$-fold product of $\Sigma$ with itself is a symplectic manifold and the products $\alpha_1\times\cdots \times \alpha_g$ and $\beta_1\times \cdots\times \beta_g$ are Lagrangian tori in this symplectic manifold. So we could try to apply Lagrangian Floer theory to these two tori, but this is not what we will do. We notice is that different labelings of the $\alpha_i$ and $\beta_i$ determine the same manifold $M$, and we would like our construction not to depend on the ordering of these curves. To this end, we consider a quotient of the $g$-fold product of $\Sigma$ by the symmetric group $S_g$ (acting by permutting the coordinates). This is called the \dfn{symmetric product} and denoted $\sym^g(\Sigma)$. We can now look at the image of the above two tori in $\sym^g(\Sigma)$ and denote them by $\mathbb{T}_\alpha$ and $\mathbb{T}_\beta$. With a little work, one can show that $\sym^g(\Sigma)$ is also a symplectic manifold and the tori $\mathbb{T}_\alpha$ and $\mathbb{T}_\beta$ are Lagrangian. Now our symplectic manifold and Lagrangian submanifolds do not depend on the ordering of the $\alpha_i$ and $\beta_i$. With a great deal of effort one can show that the Lagrangian Floer homology of $\mathbb{T}_\alpha$ and $\mathbb{T}_\beta$ is well-defined, so now we could define an invariant of the triple to be the Lagrangian Floer homology 
\[
FH(\mathbb{T}_\alpha, \mathbb{T}_\beta).
\]
This turns out to be an invariant of $M$, but not a particularly interesting one. To get an interesting invariant of $M$, we need to add one more thing. If one fixes a point $z\in \Sigma$ that is disjoint from the $\alpha_i$ and $\beta_i$ then we can consider the image of $\{z\}\times \sym^{g-1}(\Sigma)$ in $\sym^g(\Sigma)$, denoted by $F_z$. This is a symplectic submanifold, and we can modify the differential in the Floer complex to take account of this submanifold in various ways. For example, we could ignore pseudo-holomorphic strips that intersect $F_z$, or we could add a formal variable to the coefficient ring of the chain complexes and then include a power of this variable that counts the intersections of the pseudo-holomorphic strips with $F_z$. One may check that doing either of these things will still produce a chain complex, and the homology of these complexes produces a powerful invariant of $M$. (Of course, to see it is an invariant of $M$, one needs to check that different triples and different choices of $z$ lead to the same homologies.) One may find details of all of these facts above in \cite{OzsvathSzab02004} and some of the first applications in \cite{OzsvathSzab02004b}. A nice, and more details, survey of the construction can be found in \cite{OzsvathSzabo06a}.

\section{Legendrian contact homology}
We end this survey by discussing Legendrian contact homology. This is a version of Floer homology for Legendrian submanifolds of contact manifolds. It is a powerful invariant of Legendrian submanifolds and was first developed by Chekanov \cite{Chekanov02} and Eliashberg \cite{Eliashberg98} in dimension $3$, with the latter also discussing the higher-dimensional situation. They were able to show that there were Legendrian knots in $\R^3$ that had the same classical invariants but were not Legendrian isotopic because they had different Legendrian contact homology. For more details on applications of Legendrian contact homology, please see \cite{EtnyreNg2022}. The theory was developed for certain contact manifolds in higher dimensions in \cite{EkholmEtnyreSullivan05b, EkholmEtnyreSullivan07} and can be used to produce powerful (complete!) invariants of smooth knots in $\R^3$, see \cite{EkholmEtnyreNgSullivan2013}.

We begin this section by discussing a few facts about contact geometry and Legendrian submanifolds. In the next section, we define Legendrian contact homology. A key feature, and the reason why we are discussing this here, is that the compactness issues that show up in this context do not allow us to define a homology theory as a vector space, but it will need to be an algebra. So we will see that we can follow the same program to create a homology theory as we did above, but the compactness issues will dictate the algebraic structure necessary to define our chain groups.  We then end with a section discussing applications of Legendrian contact homology. 

\subsection{A little contact geometry}
There are many good introductions to contact geometry, for example \cite{Etnyre03, EtnyreTosunPre24, Geiges08}. We refer the reader to those sources for a more comprehensive introduction to the subject, but discuss a few basic definitions and results here to set the stage for Lagrangian contact homology.

A contact manifold is a pair $(M,\xi)$, where $M$ is a manifold of dimension $2n+1$, and $\xi$ is a sub-bundle of the tangent bundle $TM$, that has fibers of dimension $2n$ and is totally non-integrable. That is, $\xi$ is a hyperplane field on $M$, and being totally non-integrable means that if it is tangent to a submanifold of $M$ along an open set of the submanifold, then the submanifold has dimension less than or equal to $n$. One may use the Frobenius theorem to show that the total non-integrability is equivalent to $\xi$ being (locally) defined by a $1$-form $\alpha$ (that is $\xi=\ker \alpha$) and 
\[
\alpha\wedge (d\alpha)^n
\]
is a never zero $2n+1$-form. 
\bex
Let $(X,d\lambda)$ be an exact symplectic manifold. Then the manifold $M=X\times\R$ has a contact structure $\xi_\lambda$ given by the kernel of $dt-\lambda$, where $t$ is the coordinate on $\R$. 

A specific example of this construction is $M=T^*W\times\R$ where we use the Liouville form $\lambda_{std}$ on the cotangent bundle $T^*W$. \

An even more specific example is $\R^{2n+1}=T^*\R^n\times \R$ with the contact form $\alpha=dt-\sum_{i=1}^n y_i \, dx_i$, where the $x_i$ are coodinates on $\R^n$ and the $y_i$ are coordinates on the fibers of $T^*\R^{n}$. See Figure~\ref{fig:ctex} for the case when $n=1$.
\begin{figure}[htb]{
\begin{overpic}
{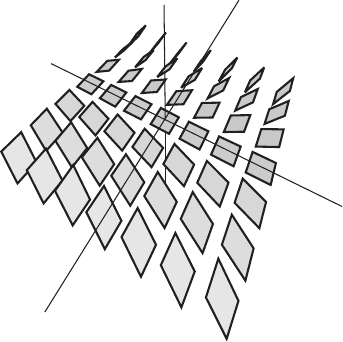}
\put(116, 165){$y$}
\put(169, 61){$x$}
\put(78, 165){$t$}
\end{overpic}}
\caption{The standard contact structure on $\R^3$.}
\label{fig:ctex}
\end{figure} 
\eex

A submanifold $L$ of a $(2n+1)$-dimensional contact manifold $(M,\xi)$ is called \dfn{Legendrian} if it has dimension $n$ and the tangent space of $L$ is contained in $\xi$: $T_xL\subset \xi_x$ for all $x\in L$. Legendrian submanifolds are essential objects in contact geometry, see for example, \cite{Etnyre05, etnyre2025pre}.

Suppose that $L$ is a Legendrian submanifold of $(X\times\R,dt-\lambda)$ and let $\pi\co (X\times\R)\to X$ be the projection map. 
\bhw
Show that $\pi(L)$ is an immersed exact Lagrangian submanifold of the symplectic manifold $(X,-d\lambda)$. 
\ehw
The idea of Legendrian contact homology is to try to ``do'' Lagrangian Floer homology with $\pi(L)$. We will see below that this does not work, but with some modification, we can build a different homology theory that does work. 


\subsection{Defining Legendrian contact homology}
As mentioned at the end of the last section, we would like to try to define some form of Lagrangian Floer homology for $\pi(L)$ where $L$ is a compact Legendrian submanifold of $(X\times\R,dt-\lambda)$ and $\pi\co (X\times\R)\to X$ is the projection map. To that end, observe that, possibly after a small perturbation, we can assume that $\pi(L)$ is embedded except for a finite number of transverse double points. We could then try to define a complex as a vector space generated by the double points of $\pi(L)$ and then, after choosing a complex structure $J$ compatible with $-d\lambda$, define a differential by counting pseudo-holomorphic strips between double points. We begin by seeing that this does not work!

We will consider the situation where $X=T^*\R$ (though the same considerations occur in the more general situation). Consider a pseudo-holomorphic map $u\co (\R\times [0,1])\to \R^2$ such that the boundary of $\R\times [0,1]$ maps to $\pi(L)$ and as one approaches infinity, the map limits to double points. Notice that this map is a pseudo-holomorphic map from $\R\times [0,1]$ to $\R^2$, so we can use our knowledge of single variable complex analysis to study this map. We first note that if there is a non-regular point of $u$ on the interior of $\R\times [0,1]$, then generically, we can assume it is of the form $z\mapsto z^n$ and, if $n>2$, by a small perturbation, the critical point will break into several critical points which have local model $z\mapsto z^2$. So, assuming our critical point has a local model of the form $z\mapsto z^2$, we note that there is a $2$-parameter family of maps near $u$ with the same image that come by moving the critical point (since $\R^2$ is $2$-dimensional). We also note that we still have the $\R$ actions on such maps, so we really have a $2$-dimensional family, even after modding out by this $\R$ action. Since we are only interested in strips in $1$-dimensional families (see the section on Lagrangian Floer homology), we can assume the maps $u$ we are interested in do not have interior critical points. 

Now consider a critical point on the boundary. Generically, we can assume it is also of the form $z\mapsto z^2$, but now the critical points must map to $\pi(L)$, and thus, there will be a $1$-parameter family of maps near $u$. So again, if we are interested in rigid maps, we will only consider maps $u$ with no critical points on the boundary or interior. That is, $u$ is an immersion. However, when trying to prove $\partial^2=0$, we will need to consider $1$-parameter familes and so will consider $u$ with a critical point on its boundary. 

Now suppose $x$, $y$, and $z$ are double points of $\pi(L)$ and $u\co (\R\times [0,1])\to \R^2$ is a rigid pseudo-holomorphic strip from $x$ to $y$ with boundary in $\pi(L)$ and $u'\co (\R\times [0,1])\to \R^2$ is a similar strip from $y$ to $z$. So if we defined our homology theory like in the case of Floer homology, then $z$ would be counted in $\partial^2 x$. We can see such a situation in Figure~\ref{fig:2disks}. 
\begin{figure}[htb]{
\begin{overpic}
{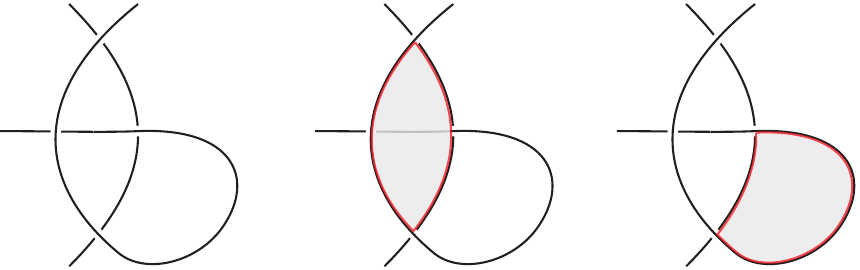}
\put(37, 109){$x$}
\put(37, 15){$y$}
\put(69, 70){$z$}
\put(17, 70){$w$}
\put(194, 50){$u$}
\put(375, 30){$u'$}
\end{overpic}}
\caption{Part of the projection $\pi(L)$ is shown on the left. The middle diagram shows the image of $u$ while the right-hand diagram shows the image of $u'$.}
\label{fig:2disks}
\end{figure} 
Figure~\ref{fig:diskfamily} shows the family of disks formed when $u$ and $u'$ are glued.
\begin{figure}[htb]{
\begin{overpic}
{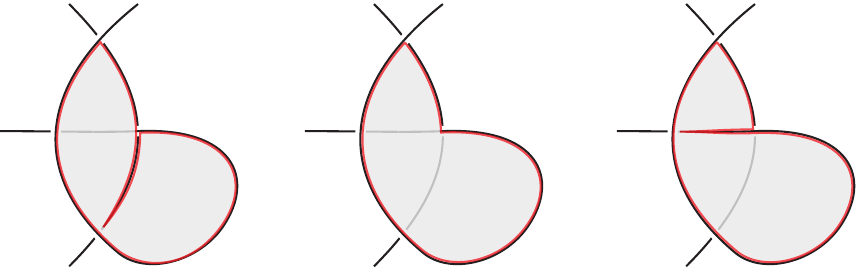}
\end{overpic}}
\caption{On the left, we see the image of a pseudo-holomorphic strip obtained by gluing $u$ and $u'$. Notice the branched point on the boudnary near the crossing $y$. In the middle, we see the middle of the family of pseudo-holomorphic strips where the ``branched point'' is ``at infinity'' or ``at $z$''. On the right, we see a pseudo-holomorphic strip with the branched point approaching $w$. }
\label{fig:diskfamily}
\end{figure} 
We notice that this family will now break into pseudo-holomorphic disks as shown in Figure~\ref{fig:broken}.
\begin{figure}[htb]{
\begin{overpic}
{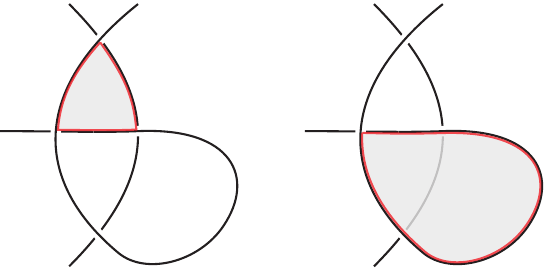}
\end{overpic}}
\caption{The two pseudo-holomorphic disks in the boundary of the family in Figure~\ref{fig:diskfamily}.}
\label{fig:broken}
\end{figure} 
So we see that the boundary of the compactified moduli spaces has a broken pseudo-holomorphic strip made out of two rigid pseudo-holomorphic strips and also a broken disk made out of a pseudo-holomorphic disk going through three double points and another pseudo-holomorphic disk going through only one double point. 

Let us formalize these sorts of pseudo-holomorphic disks. Let $p_0, \ldots, p_k$ be $k+1$ points arranged anti-clockwise round the boundary of the unit disk $D^2$ in $\R^2$, and let $D_{k+1}$ be $D^2-\{p_0,\ldots, p_k\}$. We will consider maps 
\[
u\co D_{k+1}\to \R^2
\]
with $u(\partial D_{k+1})\subset \pi(L)$. We will be interested in such maps where the limits as we approach one of the $p_i$ is a double point of $\pi(L)$. Moreover, we call a puncture $p_i$ of $u$ \dfn{positive} if as we transverse $\partial D_{k+1}$ in the anti-clockwise direction we move from a lower sheet of $\pi(L)$ to an upper sheet of $\pi(L)$ and we will call the puncture \dfn{negative} if we move from an upper sheet to a lower sheet. 

In this language, we see that if we ``glue'' two rigid pseudo-holomorphic disks from $D_2$ to $\R^2$, we can obtain a $1$-parameter family of disks whose other boundary component consists of pseudo-holomorphic disks with domain $D_3$ and $D_1$. So it is clear that the formalism in the definition of Floer homology for a pair of Lagrangian submanifolds from the last section will not work here! Moreover, it is not hard to see that if we also consider disks with $3$ punctures, then in $1$-parameter families we will also see disks with more punctures on their boundaries. We do note that if we glue a rigid pseudo-holomorphic disk with $n$ punctures, one being positive and the other negative, to a disk with $m$ punctures of the same type, then we get a $1$-parameter family of pseudo-holomorphic disks with punctures that will have boundary two disks with punctures. 

So if we try to define a differential of a double point in $\pi(L)$ we need to consider all such disks. So the differential of a double point $x$ will need to count disks with a positive puncture at $x$ and negative punctures at $x_1,\ldots, x_k$. Then the chain complex cannot be a vector space generated by the double points, but it will need to be an algebra generated by the double points so that the boundary consists of words in the double points. {\bf So we see that the compactness issues in the problem force us to choose an algebraic structure for the chain complex when trying to define a homology theory.}

Based on the above observations, we can now give the definition of Legendrian contact homology for Legendrian knots in $X\times\R$ where $X$ is an exact symplectic manifold. To this end, if $L$ is such a Legendrian submanifold, we let the chain complex $C(L)$ be the free algebra over $R$ generated by the double points of $\pi(L)$, where $R$ is the ring $\Z/2\Z$ or $Z$ (there are more sophisticated rings as well). 

We will now define a grading on $C(L)$. To simplify this, we assume that the tangent space of $X$ can be globally trivialized: $TX\cong X\times \C^{n}$. This assumption is not necessary, but without it, the definition is more complicated. Now, given a double point $x\in \pi(L)$; there are two points in $L$ that map to $x$. Call the one with the larger $\R$-coordinate $x^+$ and the other $x^-$. Choose a path $\gamma_x\co [0,1]\to L$ from $x^+$ to $x^-$. Notice that $T_{\pi(\gamma_x(t))} \pi(L)$ is a Lagrangian subspace of $T_{\pi(\gamma_x(t))}X=\C^n$. Because the tangent bundle of $X$ is trivial, we obtain a path of Lagrangians in $\C^n$. We can make this path a loop as we did in Section~\ref{maslovsection} above. The \dfn{Conley-Zehnder index} $\nu(x)$ of $x$ is simply the Maslov index of this loop of Lagrangian subspaces. Notice that this number depends on the path $\gamma$ we chose. To deal with this, we notice that we can get a loop of Lagrangian subspaces in $\C^n$ from any loop $l$ in $L$. By taking the Maslov index of the loop of Lagrangian subspaces associated to $\pi(l)$, we obtain a map $\mu\co H_1(L)\to \Z$. Let $c(L)$ be the smallest positive integer in the image of this map. It is not hard to see that $\nu(x)$ is well-defined modulo $c(L)$. Finally, we define the \dfn{grading} of a double point $x$ to be $|x|=\nu(x)-1$, and the grading of a word in double points is the sum of the gradings of the elements in the word. Our grading is well-defined modulo $c(L)$. 

To define the differential of a point $p$, we consider pseudo-holomorphic maps $u\co D_{k+1}\to X$ with $u(\partial D_{k+1})\subset \pi(L)$, with the puncture $p_0$ being positive and mapping to $x_0$ and the other punctures being negative where $p_i$ maps to $x_i$. Denote the moduli space of such maps, modulo pseudo-holomorphic reparameterization, by $\mathcal{M}^{x_0}_{x_1,\ldots, x_k}$.
\begin{tcolorbox}[title={Legendrian Floer Theorem 1 (Transversality)}]
Suppose $L$ is a Legendrian submanifold of $(X\times\R,dt-\lambda)$ such that $\pi(L)$ has transverse double points.
For a generic choice of almost complex structure $J$ compatible with $d\lambda$ or a generic perturbation of $L$, the space 
\[
\mathcal{M}^{x_0}_{x_1,\ldots, x_k}
\]
is a manifold of dimension $|x_0|-\sum_{i=1}^k |x_i|-1$.
\end{tcolorbox}
(We note that the dimension in this theorem is only well-defined modulo $c(L)$, since the grading is only defined modulo $c(L)$. This just means that the moduli space can have components of dimension $|x_0|-\sum_{i=1}^k |x_i|-1 + nc(L)$ for any integer $n$.)
\begin{tcolorbox}[title={Legendrian Floer Theorem 2 (Orientability)}]
Suppose $L$ is a Legendrian submanifold of $(X\times\R,dt-\lambda)$ such that $\pi(L)$ has transverse double points. If $L$ is spin, then
$
\mathcal{M}^{x_0}_{x_1,\ldots, x_k}
$ is orientable.
\end{tcolorbox}

For our compactness result, we need to see to what a sequence of pseudo-holomorphic disks in $\mathcal{M}^{x_0}_{x_1,\ldots, x_k}$ can converge. To this end, we define a broken pseudo-holomorphic disk. Given an element $u\in \mathcal{M}^{x_0}_{x_1,\ldots, x_k}$ and $v\in \mathcal{M}^{x_i}_{x_{k+1},\ldots, x_{l}}$ for some $i$ between $1$ and $k$ and some $l\geq k$, we call the union of $u$ and $v$ a \dfn{broken pseudo-holomorphic disks}. We can define a sequence of disks $u_n$ in $\mathcal{M}^{x_0}_{x_1,\ldots, x_{i-1}, x_{k+1},\ldots, x_{l},x_{i+1},\ldots, x_k}$ converging to the broken disk defined by $u$ and $v$ just as we did in Section~\ref{Gc}. See Figure~\ref{fig:brokendisk}.
\begin{figure}[htb]{
\begin{overpic}
{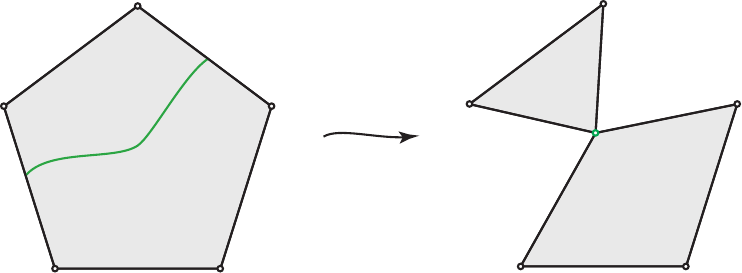}
\put(62, 119){$+$}
\put(8, 76){$-$}
\put(119, 76){$-$}
\put(30, 5){$-$}
\put(98, 5){$-$}
\put(281.5, 117){$+$}
\put(234, 81){$-$}
\put(343, 72){$-$}
\put(257, 5){$-$}
\put(321, 5){$-$}
\put(288, 59){$+$}
\put(276, 73){$-$}
\end{overpic}}
\caption{A pseudo-holomorphic disk converging to a broken pseudo-holomorphic disk.}
\label{fig:brokendisk}
\end{figure}
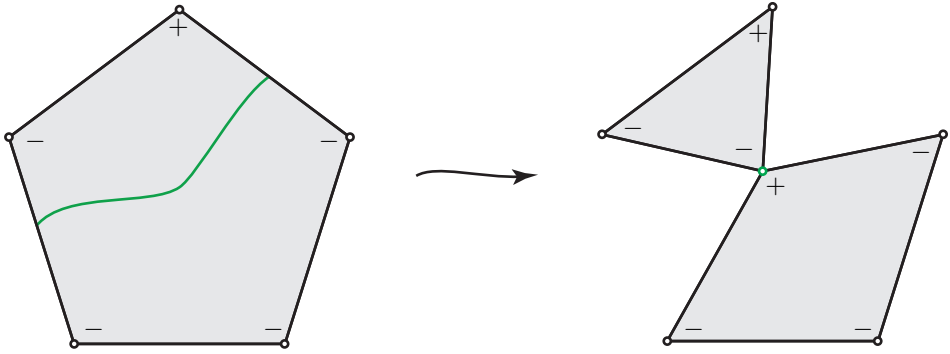 

We now have our compactness result. 
\begin{tcolorbox}[title={Legendrian Floer Theorem 3 (Compactness)}]
Suppose $L$ is a Legendrian submanifold of $(X\times\R,dt-\lambda)$ such that $\pi(L)$ has transverse double points. A sequence of pseudo-holomorphic disks in $\mathcal{M}^{x_0}_{x_1,\ldots, x_k}$ has a subsequence that converges to a pseudo-holomorphic disk in $\mathcal{M}^{x_0}_{x_1,\ldots, x_k}$ or a broken pseudo-holomorphic disk.
\end{tcolorbox}
\begin{remark}
We note that, for simplicity, this theorem is not quite accurate in some cases, as we defined $\mathcal{M}^{x_0}_{x_1,\ldots, x_k}$ to be pseudo-holomorphic disks moded out by pseudo-holomorphic reparameterization. When $k$ is large, there are no such reparameterizations and the theorem is correct as stated, but for small $k$, the statement should be made before one considers reparameterization.
\end{remark}

Just as we did while discussing Floer homology, we roughly state the gluing theorem as follows. 
\begin{tcolorbox}[title={Legendrian Floer Theorem 4 (Gluing)}]
Any broken pseudo-holomorphic disk (that is, anything that can appear in the limit in Legendrian Floer Theorem 3) can be ``glued'' to find a family of pseudo-holomorphic disks that converge to it. 
\end{tcolorbox}

We are now ready to define our differential. We will define it on generators and then extend to the algebra using linearity and a signed Leibniz rule 
\[
\partial (w_1w_2)=(\partial w_1)w_2+(-1)^{|w_1|}w_1\partial (w_2).
\]
Let $x$ be a double point of $\pi(L)$, that is, a generator of $C(L)$. We define the \dfn{differential of $x$} to be
\[
\partial x= \sum_{\substack{x_1,\ldots, x_k \\ |x|-\sum_{i=1}^k|x_i|=1}} |\mathcal{M}^{x}_{x_1,\ldots, x_k}|\,  x_1\cdots x_k,
\]
where $|\mathcal{M}^{x}_{x_1,\ldots, x_k}|$ is the signed count of the $0$ dimensional components of $\mathcal{M}^{x}_{x_1,\ldots, x_k}$, and $k$ is any non-negative integer. 

Arguing as in previous sections, we can easily check the following result. 
\begin{theorem}
The differential $\partial \co C(L)\to C(L)$ lowers the grading by $1$ and satisfies $\partial^2=0$.
\end{theorem}
We can now define the \dfn{Legendrian contact homology} to be
\[
LCH(L)=\frac{\ker(\partial\co C(L)\to C(L))}{\img(\partial\co C(L)\to C(L))}.
\]

Moreover, we can also prove the invariance of the homology under Legendrian isotopy and the choice of complex structure $J$, by using alternate versions of Legendrian Floer Theorems~1 through~4 in a similar way we proved that Morse homology is independent of the chosen metric and Morse function. 

\begin{theorem}
If $L$ and $L'$ are Legendrian isotopic, then $LCH(L)\cong LCH(L')$. 
\end{theorem}

\subsection{Applications}
It was an open question whether two Legendrian knots in the standard contact structure on $\R^3$ that were smoothly isotopic and had the same ``classical invariants" (these are invariants associated with the tangent information of a Legendrian knot) had to be Legendrian isotopic. In the late 1990s, Chekanov \cite{Chekanov02} and Eliashberg \cite{Eliashberg98} defined Legendrian contact homology and used it to show that such Legendrian knots do indeed exist. It turns out that it is very difficult to work with $LCH(L)$ since it is a non-abelian object. So it was important to be able to extract computable information from this difficult algebraic object. This was done using ``augmentations" and "linearizations". We refer the reader to \cite{EtnyreNg2022} for a recent overview of this whole theory, along with several computations. In a series of papers \cite{EkholmEtnyreSullivan05b, EkholmEtnyreSullivan05a, EkholmEtnyreSullivan05c, EkholmEtnyreSullivan07}, Ekholm, Sullivan, and the author developed this theory in higher dimensions and also used it to distinguish many Legendrian submanifolds. 

There is another interesting and surprising application of Legendrian contact homology. Specifically, one can use it to study and distinguish smooth knots in $\R^3$! This actually works for submanifolds in any dimensions, but we will restrict our discussion to knots in $\R^3$. We first construct our contact manifold. If we fix a metric on $\R^3$ and look at the unit cotangent bundle $U^*\R^3\subset T^*\R^3$, we can restrict the Liouville form $\lambda_{std}$ on $T^*\R^3$ to obtain a contact form on $U^*\R^3$. Now, if $K$ is a knot in $\R^3$, then we can consider $U^*(K)$, the unit conormal bundles
\[
U^*(K)=\{\alpha\in U_x^*\R^3: x\in K \text{ and } \alpha \text{ vanishes on $T_xK$}\}.
\]
One may easily check that $U^*(K)$ is a Legendrian torus in $U^*\R^3$; and, moreover, if $K$ and $K'$ are smoothly isotopic then $U^*(K)$ and $U^*(K')$ are Legendrian isotopic. Thus, the Legendrian contact homology of $U^*(K)$ is an invariant of the smooth isotopy class of $K$! It is known that this is almost a complete invariant (and might be a complete invariant) \cite{EkholmNgShende18}. For more on this story, see \cite{EkholmEtnyreNgSullivan2013}. 

\def\cprime{$'$}

\end{document}